\documentclass[11pt]{article}

\usepackage{amsfonts,amsmath,amsthm,amssymb}
\usepackage{bm}
\usepackage[colorlinks, citecolor=blue, urlcolor=blue, backref=page]{hyperref}

\usepackage{geometry}
\usepackage[utf8]{inputenc}
\usepackage[T1]{fontenc}    % use 8-bit T1 fonts

\usepackage{enumitem}
\usepackage{hyperref}
\usepackage{url}

\newtheorem{theorem}{Theorem}[section]

\newtheorem{lemma}[theorem]{Lemma}
\newtheorem{proposition}[theorem]{Proposition}
\newtheorem{assumption}[theorem]{Assumption}

\theoremstyle{definition}

\newtheorem{example}[theorem]{Example}
\newtheorem{remark}[theorem]{\textbf{Remark}}
\numberwithin{equation}{section}

\renewcommand{\P}{{\mathbb P}}

\newcommand{\E}{{\mathbb E}}

\newcommand{\Q}{{\mathbb Q}}
\newcommand{\R}{{\mathbb R}}

\newcommand{\N}{{\mathbb N}}

\newcommand{\Fcal}{{\mathcal F}}
\newcommand{\Gcal}{{\mathcal G}}

\newcommand{\Wcal}{{\mathcal W}}

\newcommand{\Mid}{{\ \Big|\ }}

\newcommand{\ir}{{i}}
\newcommand{\jr}{{j}}
\newcommand{\ib}{{\bm i}}
\newcommand{\jb}{{\bm j}}
\newcommand{\nG}{{G}}

\numberwithin{equation}{section}
\numberwithin{theorem}{section}

\begin{document}

\title{Propagation of chaos for maxima of particle systems with mean-field drift interaction\footnote{The authors would like to thank Dan Lacker for valuable input, in particular for communicating the elegant argument presented in Remark~\ref{R_Lacker}. Several insightful comments by an anonymous referee are gratefully acknowledged. This work has been partially supported by the National Science Foundation under grant NSF DMS-2206062.}}
\author{Nikolaos Kolliopoulos\thanks{Department of Mathematical Sciences, Carnegie Mellon University, \url{nkolliop@andrew.cmu.edu}.}
\and
Martin Larsson\thanks{Department of Mathematical Sciences, Carnegie Mellon University, \url{larsson@cmu.edu}.}
\and
Zeyu Zhang\thanks{Department of Mathematical Sciences, Carnegie Mellon University, \url{zeyuzhan@andrew.cmu.edu}.}
}
%\date{}
\maketitle

\begin{abstract}
We study the asymptotic behavior of the normalized maxima of real-valued diffusive particles with mean-field drift interaction. Our main result establishes propagation of chaos: in the large population limit, the normalized maxima behave as those arising in an i.i.d. system where each particle follows the associated McKean--Vlasov limiting dynamics. Because the maximum depends on all particles, our result does not follow from classical propagation of chaos, where convergence to an i.i.d. limit holds for any fixed number of particles but not all particles simultaneously. The proof uses a change of measure argument that depends on a delicate combinatorial analysis of the iterated stochastic integrals appearing in the chaos expansion of the Radon--Nikodym density.
\end{abstract}

%\tableofcontents

\section{Introduction}

This paper is concerned with the large-population asymptotics of the maxima of certain real-valued diffusive particle systems $X^{1,N},\ldots,X^{N,N}$ with mean-field interaction through the drifts. Specifically, we are interested in large-$N$ limits of
\begin{equation} \label{eq_normed_max}
\max_{i \le N} \frac{X^{i,N}_T - b^N_T}{a^N_T},
\end{equation}
where $a^N_T$ and $b^N_T$ are suitable normalizing constants. The particle dynamics are specified as follows, specializing the setup of \cite{JA19}. For each $N \in \N$ the $N$-particle system evolves according to a stochastic differential equation of the form
\begin{equation} \label{eq_particle_SDE}
dX^{i,N}_t = A(t, X^{i,N}_{[0,t]}) \left( B\left( t, X^{i,N}_{[0,t]}, \int g(t, X^{i,N}_{[0,t]}, \bm y_{[0,t]}) \mu^N_t(d\bm y) \right) dt + dW^i_t \right) + C(t, X^{i,N}_{[0,t]}) dt
\end{equation}
for $i=1,\ldots,N$, with i.i.d.\ initial conditions $X^{i,N}_0 \sim \nu_0$ where $\nu_0$ is a given probability measure on $\R$. We use the notation $\bm x_{[0,t]} = (x(s))_{s \in [0,t]}$ for any continuous function $\bm x$, and for each $t \in \R_+$ we let
\[
\mu^N_t = \frac1N \sum_{i=1}^N \delta_{X^{i,N}_{[0,t]}}
\]
denote the empirical measure of the particle trajectories up to time $t$. The coefficients $A(t, \bm x_{[0,t]})$, $B(t, \bm x_{[0,t]}, r)$, $C(t, \bm x_{[0,t]})$ and the interaction function $g(t, \bm x_{[0,t]}, \bm y_{[0,t]})$ are defined for all $t \in \R_+$, $\bm x, \bm y \in C(\R_+)$, and $r \in \R$. Precise assumptions are discussed below. Finally, $W^i$, $i \in \N$, is family of independent standard Brownian motions. We emphasize that there is no interaction in the volatility coefficient $A$. This is crucial for the methods used in this paper.

Under suitable assumptions, classical propagation of chaos \cite{MR221595,MR1108185,MR968996} states that for any fixed number $k \in \N$, the first $k$ particles $(X^{1,N}, \ldots, X^{k,N})$ converge jointly as $N \to \infty$ to $k$ independent copies $(X^1,\ldots,X^k)$ of the solution to the McKean--Vlasov equation
\begin{equation} \label{eq_McKV_SDE}
\begin{aligned}
dX_t &= A(t, X_{[0,t]}) \left( B\left( t, X_{[0,t]}, \int g(t, X_{[0,t]}, \bm y_{[0,t]}) \mu_t(d\bm y) \right) dt + dW_t \right) + C(t, X_{[0,t]}) dt \\
\mu_t &= \text{Law}(X_{[0,t]})
\end{aligned}
\end{equation}
with initial condition $\mu_0 = \nu_0$. A rigorous version of this statement that fits our current setup is given in \cite[Theorem~2.1]{JA19}, where convergence takes place in total variation and comes with quantitative bounds on the distance between the $k$-tuple from the $N$-particle system and the limiting $k$-tuple; see also \cite{MR3841406}.

At an intuitive level, propagation of chaos means that for large $N$ the interacting particle system behaves approximately like a system of i.i.d.\ particles. This intuition suggests that the large-$N$ asymptotics of the normalized maxima in \eqref{eq_normed_max} should match the asymptotics of the normalized maxima of the independent copies $X^i$ of the solution of \eqref{eq_McKV_SDE},
\begin{equation} \label{eq_normed_max_McKV}
\max_{i \le N} \frac{X^i_T - b^N_T}{a^N_T}.
\end{equation}
Because they are i.i.d., the latter fall within the framework of classical extreme value theory; see e.g.\ \cite{HF10, MR2364939} for an introduction. This intuition is flawed however, because propagation of chaos only makes statements about a fixed number $k$ of particles, while the maximum $\max_{i\le N} X^{i,N}_T$ depends on \emph{all} the particles. Furthermore, there are lower bounds on how similar $(X^{1,N},\ldots,X^{k,N})$ and $(X^1,\ldots,X^k)$ can be in general. In a simple Gaussian example, it is shown in \cite{LA22} that the relative entropy between the two is bounded below by a constant times $(k/N)^2$. In particular, if $k \to \infty$ and $k/N$ remains bounded away from zero, convergence does not take place. Barriers of this kind have prevented us from deriving statements about normalized maxima as corollaries of standard results on propagation of chaos. 

Our main result nonetheless shows, under assumptions, that the normalized maxima of the $N$-particle systems do behave asymptotically like those of an i.i.d.\ system. In this sense, one has \emph{propagation of chaos of normalized maxima}. The following statement is slightly informal; Theorem~\ref{T_main_precise} gives the precise version.

\begin{theorem}\label{T_main}
Suppose Assumptions~\ref{ass1} and \ref{ass2} below are satisfied. Fix $T \in (0,\infty)$ and suppose that for some normalizing constants $a^N_T,b^N_T$ the normalized maxima \eqref{eq_normed_max_McKV} of the i.i.d.\ system converge weakly to a nondegenerate distribution $\Gamma_T$ on $\R$ as $N \to \infty$. Then the normalized maxima \eqref{eq_normed_max} of the interacting particle systems also converge to $\Gamma_T$ as $N \to \infty$.
\end{theorem}

The precise assumptions are discussed in Section~\ref{S_results}, along with additional comments, and examples are developed in Section~\ref{S_examples}. Here we only highlight three points, deferring the details to Sections~\ref{S_results} and~\ref{S_examples}.

First, a key motivating example and application of Theorem~\ref{T_main} comes from a class of models known as rank-based diffusions, which were first studied by \cite{MR1894767} in the context of stochastic portfolio theory. In a rank-based model with drift interaction, the $N$-particle system evolves as
\[
dX^{i,N}_t = B\left( \frac{1}{N} \text{rank}_t(X^{i,N}_t) \right) dt + dW^i_t, \quad i=1,\ldots,N,
\]
where $\text{rank}_t(X^{i,N}_t)$ denotes the \emph{rank} of the $i$th particle within the population: $\text{rank}_t(X^{i,N}_t) = k$ if $X^{i,N}_t$ is the $k$th largest particle, with a suitable convention in case of ties. The factor $1/N$ anticipates a passage to the large-$N$ limit. Rank-based diffusions of this type have been studied extensively and their mean-field asymptotics are well understood. However, the asymptotics of the largest particle, of particular interest in the applied context, were previously unknown. As shown in Example~\ref{ex_2}, our main result is applicable and allows us to fill this gap.

Second, note that Theorem~\ref{T_main} only asserts one-dimensional marginal convergence at single time points $T$. Nonetheless, as discussed in Section~\ref{S_results}, in some cases one expects joint marginal convergence of the form
\[
\left(\max_{i \le N} \frac{X^{i,N}_{T_1} - b^N_{T_1}}{a^N_{T_1}}, \ldots, \max_{i \le N} \frac{X^{i,N}_{T_n} - b^N_{T_n}}{a^N_{T_n}} \right) \Rightarrow \Gamma_{T_1} \otimes \cdots \otimes \Gamma_{T_n} \text{ as } N \to \infty,
\]
for any $T_1 < \ldots < T_n$, where the limit is in the sense of weak convergence toward a product measure with nondegenerate components. No continuous process has finite-dimensional marginal distributions of this form, so this precludes convergence at the level of continuous processes.

Third, as part of the hypotheses of Theorem~\ref{T_main} we assume that the normalized maxima \eqref{eq_normed_max_McKV} of the i.i.d.\ system admit a nondegenerate limit law $\Gamma_T$. Classical extreme value theory asserts that up to affine transformations, $\Gamma_T$ must belong to a one-parameter family of extreme value distributions consisting of the Fr\'echet, Gumbel, and Weibull distributions. It is obviously of interest to characterize $\Gamma_T$ in terms of the data $A,B,C,g$, and $\nu_0$. This question is the subject of ongoing work, and falls outside the scope of this paper. Nonetheless, in the examples in Section~\ref{S_examples}, we are able to verify this domain of attraction hypothesis by hand.

Let us mention that the large body of work that exists on the extreme eigenvalue statistics of random matrices is related to our paper in that those eigenvalues in many cases can be described by mean-field interacting diffusions. For example, the eigenvalues of a GUE (Gaussian Unitary Ensemble) random matrix are described by Dyson Brownian motion. However, the largest eigenvalue, suitably normalized, converges in distribution to the Tracy--Widom law \cite{MR1257246}, which is different from the extreme value distributions that can arise in our framework. Another random matrix model is the Ginibre ensemble \cite{MR173726}, whose normalized spectral radius converges to the Gumbel law \cite{MR1986426}. Although this is the same limit law that we observe in our examples, the interaction among the eigenvalues of the Ginibre ensemble is not covered by our setup, in essence because we exclude interaction in the diffusion coefficients.

The rest of the paper is organized as follows. First, we finish the introduction with an outline of some of the main steps and ideas of the proof of the main theorem. Then, in Section~\ref{S_results}, we give precise statements of our assumptions and results. We also reproduce an argument due to D.~Lacker (personal communication) which vastly simplifies the proof under suitable Lipschitz assumptions; see Remark~\ref{R_Lacker}. Examples and applications are discussed in Section~\ref{S_examples}. Section~\ref{sect4} collects key lemmas needed for the proof of the main theorem. These lemmas are proved in Section~\ref{S_proofs_of_key_lemmas}. Finally, the main theorem is proved in Section~\ref{S_main_proof}. We will frequently use the notation
\[
[n] = \{1,\ldots,n\} \text{ for any } n \in \N = \{1,2,\ldots\},
\]
and $\R_+ = [0,\infty)$. We will allow generic constants $C$ to vary from line to line, and occasionally indicate the dependence on parameters by writing $C(n)$, $C(p,n)$, etc.

\subsection{Outline of the proof of Theorem~\ref{T_main}.} \label{S_proof_outline}

The remainder of this introduction contains an outline of some of the main steps and ideas of the proof of Theorem~\ref{T_main}. To simplify the discussion we take $A = 1$, $C = 0$, and $B(t, \bm x_{[0,t]}, r) = r$. We fix $T \in (0,\infty)$ and note that the theorem will be proved if we show that for any $x \in \R$,
\begin{equation} \label{eq_proof_sketch_1}
\P\left( \max_{i \le N} \frac{X^{i,N}_T - b^N_T}{a^N_T} \le x \right) - \P\left( \max_{i \le N} \frac{X^i_T - b^N_T}{a^N_T} \le x \right) \to 0 \text{ as } N \to \infty.
\end{equation}
Here $X^i$, $i \in \N$, are i.i.d.\ copies of the solution of \eqref{eq_McKV_SDE} with driving Brownian motions $W^i$, and all objects are defined on a filtered probability space $(\Omega,\Fcal,(\Fcal_t)_{t\ge0},\P)$.

The first observation, going back at least to \cite{MR1080535} and also used by \cite{MR885876,MR3841406,JA19}, is that the structure of the particle dynamics \eqref{eq_particle_SDE} allows us to construct for each $N$ a (locally) equivalent measure $\Q^N \sim_\text{loc} \P$ under which $(X^1,\ldots,X^N)$ acquires the law of $(X^{1,N},\ldots,X^{N,N})$. This is accomplished by the Radon--Nikodym density process
\begin{equation} \label{eq_ZN}
Z^N_t = \frac{d\Q^N|_{\Fcal_t}}{d\P|_{\Fcal_t}} = \exp\left( M^N_t - \frac12 \langle M^N \rangle_t \right),
\end{equation}
where the local martingale $M^N$ is given by
\[
M^N_t = \sum_{i=1}^N \int_0^t \int g(s, X^i_{[0,s]}, \bm y_{[0,s]}) \left( \mu^N_s(d\bm y) - \mu_s(d\bm y)\right) dW^i_s
\]
and where, by overloading notation, we set $\mu^N_t = \frac1N \sum_{i=1}^N \delta_{X^i_{[0,t]}}$. We may then re-express the left-hand side of \eqref{eq_proof_sketch_1} as
\begin{equation} \label{eq_proof_sketch_2}
\E \left[ \prod_{i=1}^N \bm1_{\{ X^i_T \le x_N \}} \left(Z^N_T - 1\right) \right], \quad \text{where} \quad x_N = a^N_T x + b^N_T.
\end{equation}
%The key point is that \eqref{eq_proof_sketch_2} is expressed in terms of the processes $(X^i,W^i)$, $i \in [N]$, which are mutually independent. Our task is to exploit this independence to show that \eqref{eq_proof_sketch_2} tends to zero in the large-$N$ limit.
The key point is that \eqref{eq_proof_sketch_2} is expressed in terms of the mutually independent processes $(X^i,W^i)$, $i \in [N]$, while the dependence that exists among the particle in the original $N$-particle system is captured by the Radon--Nikodym derivative $Z^N_T$. The proof of the theorem rests on a detailed analysis of how $Z^N_T$ interacts with the indicators in \eqref{eq_proof_sketch_2}, ultimately allowing us to ``extract enough independence'' to show that \eqref{eq_proof_sketch_2} tends to zero in the large-$N$ limit. An analogous strategy of ``extracting independence'' through the above change of measure was used in \cite{JA19}, although the actual execution of this strategy is very different in our context.

Iterating the SDE satisfied by the stochastic exponential $Z^N$ leads to the formal chaos expansion
\begin{equation} \label{eq_chaos}
Z^N_T = 1 + \sum_{m=1}^\infty  \int_0^T \int_0^{t_1} \cdots \int_0^{t_{m-1}} dM^N_{t_m} \cdots dM^N_{t_1}.
\end{equation}
If $T$ is sufficiently small, one can show that a truncated version of this expansion can be substituted for $Z^N_T-1$ in \eqref{eq_proof_sketch_2} at the cost an arbitrarily small error $\varepsilon > 0$. Importantly, although the truncation level $m_0$ (say) depends on $\varepsilon$, it does not depend on $N$. We are thus left with showing that each of the remaining $m_0$ terms tends to zero, that is, for each $m \in [m_0]$,
\begin{equation} \label{eq_proof_sketch_3}
\E \left[ \prod_{i=1}^N \bm1_{\{ X^i_T \le x_N \}} \int_0^T \int_0^{t_1} \cdots \int_0^{t_{m-1}} dM^N_{t_m} \cdots dM^N_{t_1} \right] \to 0 \text{ as } N \to \infty.
\end{equation}
This is done by substituting $\bm1_{\{ X^i_T \le x_N \}} = 1 - \bm1_{\{ X^i_T > x_N \}}$ and expanding the product, as well as substituting the definition of $M^N$ into the iterated integral and expand using multilinearity. The result is a sum consisting of all terms of the form
\begin{equation} \label{eq_proof_sketch_4}
\frac{1}{N^m} \E \left[ \prod_{r=1}^k \bm1_{\{ X^{\ell_r}_T > x_N \}} \int_0^T G^{i_1 j_1}_{t_1} \int_0^{t_1} G^{i_2 j_2}_{t_2} \cdots \int_0^{t_{m-1}} G^{i_m j_m}_{t_m} dW^{i_m}_{t_m} \cdots dW^{i_1}_{t_1} \right]
\end{equation}
with $k\in[N]$, $\{\ell_1,\ldots,\ell_k\} \subset [N]$, $(i_1,\ldots,i_m) \in [N]^m$, and $(j_1,\ldots,j_m) \in [N]^m$, and where the processes
\[
G^{ij}_t = g(t, X^i_{[0,t]}, X^j_{[0,t]}) - \int g(t, X^i_{[0,t]}, \bm y_{[0,t]}) \mu_t(d\bm y)
\]
arise when the empirical measure $\mu^N_t$ is substituted into the definition of $M^N$. We are now in a position to sketch the main ways in which we exploit the independence among the processes $(X^i,W^i)$, $i \in \N$.

For each $k$, there are $N^{2m} \binom{N}{k}$ terms of the form \eqref{eq_proof_sketch_4}. Using iterated stochastic integral estimates, along with the independence of the $X^i$, $i \in \N$, and fact that $\P(X_T > x_N) = O(1/N)$ due to the domain of attraction assumption, we show that each of these terms is bounded by
\[
\P(X_t > x_N)^{k \left(1-\frac{1}{\log N}\right)} \lceil \log N \rceil^{m} = \left( \frac{O(1)}{N} \right)^k \lceil \log N \rceil^{m}.
\]
This is not enough to deduce \eqref{eq_proof_sketch_3} however, because it only produces the upper bound $O(N^m \lceil \log N \rceil^{m})$ which does not tend to zero with $N$. Nonetheless, a refined analysis shows that a large number of the terms \eqref{eq_proof_sketch_4} are in fact zero. Very roughly, this happens when there is a small overlap between the indices $\{\ell_1,\ldots,\ell_k\}$ and $\{i_1,\ldots,i_m,j_1,\ldots,j_m\}$, in which case the expectation in \eqref{eq_proof_sketch_4} vanishes despite the presence of the indicators. A counting argument then shows that for each $k$, at most $\binom{N}{k} k (k+1) \cdots (k+m) N^{m-1}$ terms remain. Using the earlier estimate to control these remaining terms finally yields the bound $O(N^{-1} \lceil \log N \rceil^{m})$ of the left-hand side of \eqref{eq_proof_sketch_3}. This does tend to zero as $N \to \infty$ and allows us to complete the proof.

Counting the nonzero terms \eqref{eq_proof_sketch_4} and bounding their size constitute the heart of the proof. The key arguments involved are given as lemmas in Section~\ref{sect4}. However, other parts of the proof also require substantial technical effort. In particular, work is required to (i) reduce from the case of general coefficients $A,B,C$ to the simpler ones discussed above; (ii) obtain sufficiently strong iterated integral bounds to truncate the chaos expansion independently of $N$ when $T$ is small; and (iii) remove the smallness requirement on $T$. This leads to added complexity and explains why the full proof of Theorem~\ref{T_main} is rather long and technical.

\section{Assumptions and main results} \label{S_results}

To give a precise description of our setup, we first introduce regularity and growth assumptions on the data $A,B,C,g$.

\begin{assumption}\label{ass1}
The coefficient functions $(t,\bm x) \mapsto A(t, \bm x_{[0,t]})$, $(t,\bm x,r) \mapsto B(t, \bm x_{[0,t]}, r)$, $(t,\bm x) \mapsto C(t, \bm x_{[0,t]})$ and the interaction function $(t,\bm x,\bm y) \mapsto g(t, \bm x_{[0,t]}, \bm y_{[0,t]})$
are real-valued measurable functions on $\R_+ \times C(\R_+)$, $\R_+ \times C(\R_+) \times \R$, $\R_+ \times C(\R_+)$, and $\R_+ \times C(\R_+) \times C(\R_+)$, respectively. They satisfy the following conditions:
\begin{itemize}
\item $A$ and $C$ are uniformly bounded,
\item for every $t \in \R_+$ and $\bm x \in C(\R_+)$, the function $r \mapsto B(t, \bm x_{[0,t]}, r)$ is twice continuously differentiable, and its first and second derivatives are bounded uniformly in $(t, \bm x)$.
\end{itemize}
\end{assumption}

\begin{remark}
Note that $r \mapsto B(t, \bm x_{[0,t]}, r)$ itself need not be bounded, only its first two derivatives. We thus cover examples with linear growth. Moreover, if the interaction function $g$ is uniformly bounded, the growth properties of $r \mapsto B(t, \bm x_{[0,t]}, r)$ become irrelevant.
\end{remark}

By imposing further conditions we could appeal to known results on well-posedness of McKean--Vlasov equations to assert that \eqref{eq_McKV_SDE} has a solution. Rather than doing this, we will assume existence directly (uniqueness is not actually required, so we do not assume it.)

\begin{assumption}\label{ass2}
Fix a probability measure $\nu_0$ on $\R$ and assume that the McKean--Vlasov equation \eqref{eq_McKV_SDE} admits a weak solution $(X,W)$ with $X_0 \sim \nu_0$. Construct (for instance as a countable product) a filtered probability space $(\Omega,\Fcal,(\Fcal_t)_{t\ge0},\P)$ with a countable sequence $(X^i,W^i)$, $i \in \N$ of independent copies of $(X,W)$. Then, assume that there is a continuous function $K(t)$ such that for all $p \in \N$, $t \in \R_+$, $N \in \N$, and $i,j \in [N]$, one has the moment bounds
\begin{equation}\label{part1ass2}
\E\left[ g(t, X^i_{[0,t]}, X^j_{[0,t]})^{2p} \right] \le p!\, K(t)^p
\end{equation}
and
\begin{equation}\label{part2ass2}
\E\left[ \left( \int g(t, X^i_{[0,t]}, \bm y_{[0,t]}) ( \mu^N_t - \mu_t)(d\bm y) \right)^{2p} \right] \le \frac{1}{N^p} p! \, K(t)^p.
\end{equation}
\end{assumption}

Sufficient conditions for the moment bounds \eqref{part1ass2}--\eqref{part2ass2} along with further discussion are given in Remark~\ref{R_ass2} below.

Let Assumptions~\ref{ass1} and~\ref{ass2} be in force. For each $N \in \N$ we now use the processes $(X^i,W^i)$ to construct the $N$-particle systems by changing the probability measure. First define the $N$-particle empirical measure
\begin{equation} \label{eq_muN}
\mu^N_t = \frac1N \sum_{i=1}^N \delta_{X^i_{[0,t]}}.
\end{equation}
Next, define the (candidate) density process $Z^N = \exp(M^N - \frac12\langle M^N \rangle)$ where
\begin{equation} \label{eq_MN}
M^N_t = \sum_{i=1}^N \int_0^t \Delta B^{i,N}_s dW^i_s
\end{equation}
and
\begin{equation} \label{eq_DeltaBN}
\begin{aligned}
\Delta B^{i,N}_t &= B\left(t,X_{[0,t]}^i,\int g(t, X_{[0,t]}^i, \bm{y}_{[0,t]}) \mu_t^N(d\bm{y})\right) \\
& \qquad - B\left(t,X_{[0,t]}^i,\int g(t, X_{[0,t]}^i, \bm{y}_{[0,t]}) \mu_t(d\bm{y})\right).
\end{aligned}
\end{equation}
Assumptions~\ref{ass1} and~\ref{ass2} imply that $\E[ \int_0^t (\Delta B^{i,N}_s)^2 ds] < \infty$ for all $i$ and $t$, which ensures that $M^N$ is a well-defined positive martingale. We claim that $\E[Z^N_T] = 1$ for all $T \in (0,\infty)$, so that $Z^N$ is a true martingale. To see this, note that Lemma~\ref{tailest} implies that for any $s < t \le T$ with $t-s$ small enough, the chaos expansion
\[
\frac{Z^N_t}{Z^N_s} = 1 + \sum_{m=1}^\infty \int_s^t \int_s^{t_1} \cdots \int_s^{t_{m-1}} dM^N_{t_m} \cdots dM^N_{t_1}
\]
converges in $L^2$. Moreover, \cite[Proposition~1]{CAKR91} together with Assumption~\ref{ass2} imply that each iterated integral has expectation zero. As a result, $\E[Z^N_t / Z^N_s] = 1$ for all such $s,t$, and this implies $\E[Z_T]=1$ as claimed.

%Assumptions~\ref{ass1} and~\ref{ass2} imply that $\E[ \int_0^t (\Delta B^{i,N}_s)^2 ds] < \infty$ for all $i$ and $t$, which ensures that $M^N$ is a well-defined positive martingale. However, $Z^N$ is potentially not a true martingale, but if it is, 

Since $Z^N$ is a true martingale, it induces a locally equivalent probability measure $\Q^N \sim_\text{loc} \P$ under which the processes defined by
\[
W^{i,N}_t = W^i_t  - \int_0^t \Delta B^{i,N}_s ds
\]
are mutually independent standard Brownian motions. Thus under $\Q^N$ we find that $X^1,\ldots,X^N$ follow the $N$-particle dynamics \eqref{eq_particle_SDE},
\begin{equation} \label{eq_Q_particle_SDE}
dX^i_t = A(t, X^i_{[0,t]}) \left( B\left( t, X^i_{[0,t]}, \int g(t, X^i_{[0,t]}, \bm y_{[0,t]}) \mu^N_t(d\bm y) \right) dt + dW^{i,N}_t \right) + C(t, X^i_{[0,t]}) dt.
\end{equation}
%Our final assumption is that $Z^N$ is indeed a true $\P$-martingale. This makes the change of measure legitimate and ensures that the interacting $N$-particle dynamics constructed above is well-defined. Necessary and sufficient conditions are discussed in Remark~\ref{R_ass3} below.

%\begin{assumption}\label{ass3}
%For each $N \in \N$ the process $Z^N$ is a true $\P$-martingale.
%\end{assumption}

The following is the precise formulation of our main result.

\begin{theorem}\label{T_main_precise}
Suppose Assumptions~\ref{ass1} and \ref{ass2} are satisfied and consider the laws $\Q^N$ constructed above. Fix $T \in (0,\infty)$ and suppose that for some normalizing constants $a^N_T,b^N_T$ the normalized maxima of the i.i.d.\ system converge weakly to a nondegenerate distribution function $\Gamma_T$ on $\R$:
\[
\P\left( \max_{i \le N} \frac{X^i_T - b^N_T}{a^N_T} \le x \right) \to \Gamma_T(x) \text{ as } N \to \infty, \quad x \in \R.
\]
Then the normalized maxima of the interacting particle systems also converge to $\Gamma_T$:
\begin{equation} \label{eq_T_main_precise_1}
\Q^N\left( \max_{i \le N} \frac{X^i_T - b^N_T}{a^N_T} \le x \right) \to \Gamma_T(x) \text{ as } N \to \infty, \quad x \in \R.
\end{equation}
\end{theorem}

Classical extreme value theory asserts that up to affine transformations, $\Gamma_T$ must belong to a one-parameter family of extreme value distributions consisting of the Fr\'echet, Gumbel, and Weibull distributions. Our assumptions tend to preclude the heavy-tailed behavior that is characteristic of the Fr\'echet class.

\begin{proposition}\label{P_not_Frechet}
Let the assumptions of Theorem~\ref{T_main_precise} be satisfied. Assume in addition that all moments of $\nu_0$ are finite and one has the linear growth bound $|B(t,\bm x_{[0,t]}, 0)| \le c(1 + x^*_t)$ for all $t \in [0,T]$, $\bm x \in C(\R_+)$, where the constant $c$ may depend on $T$ and we use the notation $x^*_t = \sup_{s\le t}|x_s|$. Then $\Gamma_T$ must belong to the Gumbel or Weibull family.
\end{proposition}

\begin{proof}
We allow $c$ to change from one occurrence to the next. The assumptions imply the bound $|B(t, X_{[0,t]}, r)| \le c(1 + X^*_t + |r|)$, which together with the uniform boundedness of $A$ and $C$ yields
\[
|X_t| \le |X_0| + c\int_0^t \left( 1 + X^*_s + \int \left| g(s, X_{[0,s]}, \bm y_{[0,s]}) \right| \mu_s(d\bm y) \right) ds + \left| \int_0^t A(s, X_{[0,s]}) dW_s \right|.
\]
This in turn implies
\[
X^*_t \le c \int_0^t X^*_s ds + J_t
\]
for the nondecreasing process
\[
J_t = |X_0| + c\int_0^t \left( 1 + \int \left| g(s, X_{[0,s]}, \bm y_{[0,s]}) \right| \mu_s(d\bm y) \right) ds + \sup_{s \le t} \left| \int_0^s A(u, X_{[0,u]}) dW_u \right|.
\]
Pathwise application of Gronwall's inequality then yields $X^*_T \le e^{cT} J_T$. Because all moments of $\nu_0$ are finite, $A$ is uniformly bounded, and thanks to \eqref{part1ass2} of Assumption~\ref{ass2}, all moments of $J_T$ are finite. (For the stochastic integral term this uses the BDG inequalities.) Then so are the moments of $X^*_T$, and then also of $X_T$. However, if $X_T$ were in the Fr\'echet domain of attraction it would have a regularly varying tail (see \cite[Theorem~1.2.1]{HF10}), implying that all sufficiently high moments are infinite. This excludes the Fr\'echet family.
\end{proof}

\begin{remark}
Weak convergence is equivalent to convergence for all $x \in \R$ where $\Gamma_T$ is continuous. However, since all extreme value distributions are continuous, restricting to continuity points is redundant.
\end{remark}

Theorem~\ref{T_main_precise} asserts one-dimensional marginal convergence at single time points $T$. We do not prove full finite-dimensional marginal convergence in this paper, but let us nonetheless make the following observation. In certain examples, the random vectors $(X_{T_1},\ldots,X_{T_n})$ with $T_1<\ldots<T_n$ exhibit \emph{asymptotic independence}. This means that each $X_{T_\alpha}$, $\alpha \in [n]$, belongs to the maximum domain of attraction of some extreme value distribution $\Gamma_{T_\alpha}$ with normalizing constants $a^N_{T_\alpha},b^N_{T_\alpha}$, and that the vector of normalized maxima converges to a product measure:
\begin{equation} \label{eq_as_indep}
\P\left(\max_{i \le N} \frac{X^i_{T_1} - b^N_{T_1}}{a^N_{T_1}} \le x_1, \ldots, \max_{i \le N} \frac{X^i_{T_n} - b^N_{T_n}}{a^N_{T_n}} \le x_n \right) \to \Gamma_{T_1}(x_1) \cdots \Gamma_{T_n}(x_n)
\end{equation}
as $N \to \infty$ for all $(x_1,\ldots,x_n) \in \R^n$. Asymptotic independence is characterized by the condition
\[
\frac{\P (X_{T_\alpha} > a^N_{T_\alpha} x_\alpha + b^N_{T_\alpha} \text{ and } X_{T_\beta} > a^N_{T_\beta} x_\beta + b^N_{T_\beta} )}{\P(X_{T_\alpha} > a^N_{T_\alpha} x_\alpha)} \to 0
\]
for all $\alpha \ne \beta$ in $[n]$ and all $x_\alpha,x_\beta \in \R$ such that $\Gamma_{T_\alpha}(x_\alpha) > 0$ and $\Gamma_{T_\beta}(x_\beta) > 0$; see \cite[Proposition~5.27]{MR2364939}. In particular, this is known to hold for multivariate Gaussian distributions with correlation in $(-1,1)$; see \cite[Corollary~5.28]{MR2364939}. Thus if $X$ is a Gaussian process with non-trivial correlation function, then all finite-dimensional marginal distributions of the centered and scaled processes $\max_{i \le N} (X^i_t - b^N_t)/a^N_t$ converge as $N \to \infty$ to product distributions with nondegenerate components (specifically, affine transformations of Gumbel). No continuous process has finite-dimensional marginal distributions of this form, so this precludes convergence at the level of continuous processes. The Gaussian case is discussed further in Example~\ref{ex_1}. Whenever the i.i.d.\ particles $X^i$, $i \in \N$, satisfy the asymptotic independence property \eqref{eq_as_indep}, it is natural to expect that the same is true for the interacting $N$-particle systems, although proving this is outside the scope of this paper.

We end this section with a few additional remarks. % regarding Assumption~\ref{ass2}. %In particular, we give directly checkable conditions under which they are satisfied.

\begin{remark}[on Assumption~\ref{ass2}] \label{R_ass2}
There is a large literature on well-posedness of McKean--Vlasov equations, providing a range of conditions under which a solution to \eqref{eq_McKV_SDE} exists; see e.g.\ \cite{MR221595,MR1108185,MR968996,MR3841406}. Next, the moment bound \eqref{part1ass2} is satisfied if the centered random variables $g(t, X^i_{[0,t]}, X^j_{[0,t]}) - \int g(t, X^i_{[0,t]}, \bm{y})\mu_t(d\bm{y})$ are bounded or conditionally (on $X^i_{[0,t]}$) sub-Gaussian with a uniformly bounded variance proxy (see e.g.\ \cite{HDS15} for a review of sub-Gaussianity). One can then also verify \eqref{part2ass2} by noticing that the $2p$-th moment of
\begin{align*}
\int g(t, X^i_{[0,t]}, \bm{y}) (\mu_t^{N} - \mu_t )(d\bm{y}) = \frac{1}{N}\sum_{j = 1}^{N}\left( g(t, X^i_{[0,t]}, X^j_{[0,t]})  - \int g(t, X^i_{[0,t]}, \bm{y})\mu_t(d\bm{y}) \right)
\end{align*}
can be controlled by that of
\begin{equation} \label{eq_R_Ass2}
\frac{1}{N - 1}\sum_{\substack{j = 1 \\ j \neq i}}^{N}\left( g(t, X^i_{[0,t]}, X^j_{[0,t]})  - \int g(t, X^i_{[0,t]}, \bm{y})\mu_t(d\bm{y}) \right)
\end{equation}
plus a term proportional to $p! 3^p K(r)^p / N^{2p}$. % (e.g using $(a + b)^p \leq 2^{p}(a^p + b^p)$),
Conditionally on $X^i_{[0,t]}$, the $N - 1$ summands in \eqref{eq_R_Ass2} are pairwise independent and identically distributed with zero mean. In the sub-Gaussian case, these $N-1$ summands above are also sub-Gaussian, so their average is sub-Gaussian with an $O(\frac{1}{N})$ variance, and the desired bound follows from \cite[Lemma~1.4]{HDS15}. In the case of bounded summands, we instead apply Hoeffding’s inequality \cite[Theorem~1.9]{HDS15} and then again \cite[Lemma~1.4]{HDS15}.
\end{remark}

\begin{remark}[Non-i.i.d.\ initial conditions]
Standard propagation of chaos is frequently formulated under weaker assumptions on the initial conditions of the $N$-particle systems than being i.i.d. A common assumption is that $(X_0^{1,N}, \ldots, X_0^{k,N})$ converges weakly to $(X_0^{1}, \ldots, X_0^{k})$ as $N \to \infty$ for each $k \in \mathbb{N}$, where $X^i_0$, $i \in \N$, is an i.i.d.\ sequence. Although we have not succeeded in proving our main result under this weaker assumption on the initial conditions, it is nonetheless possible to move slightly beyond the i.i.d.\ setting through an additional change of measure. Specifically, let $\nu^N_0$ (a probability measure on $\R^N$) be the desired joint initial law of the $N$-particle system, and assume it is absolutely continuous with respect to the $N$-fold product measure $\nu_0^{\otimes N}$, where as above $\nu_0$ is the initial law of the limiting McKean--Vlasov SDE. We make the total variation type stability assumption that
\[
\lim_{N \to \infty} \int_{\R^N} \left| \frac{d\nu^N_0}{d\nu_0^{\otimes N}} - 1 \right|d\nu_0^{\otimes N} = 0.
\]
Letting $\Q^N$ be defined as before, we now obtain a new measure $\widetilde\Q^N$ by using
\[
\widetilde Z^N_0 = \frac{d\nu^N_0}{d\nu_0^{\otimes N}}(X^1_0,\ldots,X^N_0)
\]
as Radon--Nikodym derivative. This affects the initial law, but not the form of the particle dynamics. Then as $N \to \infty$ we have
\begin{align*}
&\left| \widetilde \Q^N\left( \max_{i \le N} \frac{X^i_T - b^N_T}{a^N_T} \le x \right) - \Q^N\left( \max_{i \le N} \frac{X^i_T - b^N_T}{a^N_T} \le x \right) \right| \\
&\qquad =
\left| \E \left[ \prod_{i=1}^N \bm1_{\{ X^{i}_{T} \le x_N \}} \left(\widetilde{Z}_0^N - 1\right) \right] \right| \le \int_{\R^N} \left| \frac{d\nu^N_0}{d\nu_0^{\otimes N}} - 1 \right| d\nu_0^{\otimes N} \to 0.
\end{align*}
This shows that the large-$N$ asymptotics of the normalized maxima of the $N$-particle system are unaffacted when the initial distribution is $\nu^N_0$ instead of $\nu_0^{\otimes N}$.
\end{remark}

\begin{remark}[A coupling argument] \label{R_Lacker}
D.\ Lacker has pointed out to us that a simple coupling argument yields our propagation of chaos result in the presence of constant volatility and Lipschitz drift. Although this does not lead to a proof of our main result (in particular, our key example of rank-based models is excluded due to discontinuous drifts; see Example~\ref{ex_2}), it is worth recording the argument here.
%Although this is not a widely applicable approach, with the key example of rank-based models (see Example~\ref{ex_2}) being excluded due to the drifts being discontinuous, it is worth recording the argument here.
Assume that the drift function $B$ satisfies the Lipschitz condition
\[
|B(x,\mu) - B(y,\nu)| \le C(|x-y| + \Wcal_p(\mu,\nu))
\]
for some constant $C$, all $x,y \in \R$, and all probability measures $\mu,\nu$ with finite $p$-th moment. Here $\Wcal_p(\mu,\nu)$ is the $p$-Wasserstein distance between $\mu$ and $\nu$ for some fixed $p \in [1,\infty)$. We let the $N$-particle system be given as the unique strong solution of the system of SDEs
\[
dX^{i,N}_t = B(X^{i,N}_t, \mu^N_t) dt + dW^i_t, \quad X^{i,N}_0 = \xi^i, \quad i=1,\ldots,N,
\]
where $W^i$, $i \in \N$, is a sequence of independent standard Brownian motions and $\xi^i$, $i \in \N$, is a sequence of $p$-integrable i.i.d.\ initial conditions. For each $i$, let $X^i$ be the unique strong solution of the McKean--Vlasov SDE
\[
dX^i_t = B(X^i_t, \mu_t)dt + dW^i_t, \quad X^i_0 = \xi^i, \quad \mu_t = \text{Law}(X^i_t),
\]
using the same Brownian motion and initial condition as for the $N$-particle systems. We then obtain
\begin{align*}
\max_{i \leq N}|X_t^{i, N} - X_t^{i}| &=  \max_{i \leq N}\left|\int_0^t \left( B(X_s^{i, N}, \mu_s^N) - B(X_s^{i}, \mu_s) \right) ds\right| \nonumber \\
&\le C\int_0^t\max_{i \leq N}|X_s^{i, N} - X_s^{i}|ds + C\int_0^T \Wcal_p(\mu_{s}^N,\mu_{s})ds,
\end{align*}
and a Gronwall-type argument gives
\begin{equation*}
\max_{i \leq N}\left|X_T^{i, N} - X_T^{i}\right| \leq Ce^{CT}\int_0^TW_p(\mu_{s}^N,\mu_{s})ds.
\end{equation*}
Consequently,
\begin{align*}
\E\left[\left|\max_{i \le N} \frac{X^{i,N}_T - b^N_T}{a^N_T} - \max_{i \le N} \frac{X^{i}_T - b^N_T}{a^N_T}\right|\right]
&\le \frac{1}{a^N_T}\E\left[\max_{i \leq N}\left|X_T^{i, N} - X_T^{i}\right|\right] \\
&\le \frac{Ce^{CT}}{a^N_T}\int_0^T\E\left[W_p(\mu_{s}^N, \mu_{s})\right]ds.
\end{align*}
Provided that $\int_0^T\E\left[W_p(\mu_{s}^N, \mu_{s})\right]ds / a^N_T \to 0$ as $N \to \infty$, our propagation of chaos result follows. This happens, for instance, in the Gaussian case where $a^N_T$ behaves like $1/\sqrt{\log N}$ (see Example~\ref{ex_1} below) and $\E\left[W_p(\mu_{s}^N, \mu_{s})\right]$ behaves like $N^{-\gamma}$ for some $\gamma > 0$.
\end{remark}

\section{Examples} \label{S_examples}

We discuss two examples that illustrate the main result.

\begin{example}[Gaussian particles] \label{ex_1}
The following Gaussian particle system has been studied in a number of contexts, such as models for monetary reserves of banks \cite{CFS13,BOCA16}, and default intensities in large interbank networks \cite[Example~2.2]{JK21}. The $N$-particle system evolves according to the multivariate Ornstein--Uhlenbeck process
\begin{align*}
X_{t}^{i, N} &= X_{0}^{i} - \kappa \int_{0}^{t}\left(X_{s}^{i, N} -  \frac{1}{N}\sum_{j=1}^NX_{s}^{j, N} \right)ds + \sigma W_t^{i}, \quad i=1,\ldots,N,
\end{align*}
with i.i.d.\ $N(m_0,\sigma_0^2)$ initial conditions. Here $\kappa, m_0 \in \R$ and $\sigma, \sigma_0 \in (0,\infty)$ are parameters. In our setting this example arises by taking $A(t,\bm x_{[0,t]}) = \sigma$, $B(t,\bm x_{[0,t]},r) = -\kappa (x_t - r)/\sigma$, $C(t, \bm x_{[0,t]}) = 0$, and $g(t,\bm x_{[0,t]},\bm y_{[0,t]}) = y_t$. Clearly Assumption~\ref{ass1} is satisfied. The McKean--Vlasov equation \eqref{eq_McKV_SDE} reduces to
\[
dX_t = - \kappa \left(X_t - \E[X_t] \right)dt + \sigma dW_t.
\]
Taking expectations one obtains $\E[X_t] = \E[X_0] = m_0$ for all $t \in \R_+$, showing that $X$ is an Ornstein--Uhlenbeck process with constant mean $m_0$ and time-$t$ variance given by
\[
\sigma_t^2 = \text{Var}(X_t) = e^{-2\kappa t}\sigma_0^2 + (1-e^{-2\kappa t}) \frac{\sigma^2}{2\kappa }.
\]
Letting $X^i$, $i \in \N$, be independent copies of $X$, we see that $g(t, X^i_{[0,t]}, X^j_{[0,t]}) = X^j_t$ is Gaussian for all $i,j$. Thus in view of Remark~\ref{R_ass2}, Assumption~\ref{ass2} is satisfied. %Finally, the $N$-particle system itself is an multivariate Ornstein--Uhlenbeck process and in particular non-explosive, so Assumption~\ref{ass3} is satisfied thanks to Remark~\ref{R_ass3}.

Now, it is a well-known fact \cite[Example~1.1.7]{HF10} that the standard Gaussian distribution belongs to the maximum domain of attraction of the standard Gumbel distribution $\Gamma(x) = \exp(-e^{-x})$ with normalizing constants
\[
b^N = \frac{1}{a^N} = \sqrt{ 2\log N - \log \log N - \log(4\pi) }.
\]
By normalizing, we see that $X_T$ also belongs to the maximum domain of attraction of $\Gamma$ for each $T$, with normalizing constants
\[
a^N_T = \sigma_T a^N \text{ and } b^N_T = m_0 + \sigma_T b^N.
\]
Indeed, since $(X^i_T - m_0)/\sigma_T$ is standard Gaussian we have
\[
\P\left( \max_{i \le N} \frac{X^i_T - b^N_T}{a^N_T} \le x \right) = \P\left( \max_{i \le N} \frac{(X^i_T - m_0)/\sigma_T - b^N}{a^N} \le x \right) \to \exp(-e^{-x}) \text{ as } N \to \infty.
\]
This shows that the hypotheses of Theorem~\ref{T_main_precise} are satisfied, and we deduce that same asymptotics hold for the $N$-particle systems,
\[
\P\left( \max_{i \le N} \frac{X^{i,N}_T - b^N_T}{a^N_T} \le x \right) \to \exp(-e^{-x}) \text{ as } N \to \infty.
\]
Lastly, $X$ is a Gaussian process with correlation function
\[
\text{Corr}(X_s,X_t) = \sqrt{ \frac{\alpha + e^{2\kappa s} - 1}{\alpha + e^{2\kappa t} - 1} }, \quad s < t,
\]
where $\alpha = 2 \kappa \sigma_0^2 / \sigma^2$. Thus $\text{Corr}(X_s,X_t) \in (0,1)$ for all $s \ne t$. The discussion in Section~\ref{S_results} implies that the finite-dimensional marginal distributions of $X$ exhibit asymptotic independence, and that \eqref{eq_as_indep} holds for all $n \in \N$ and $T_1 < \cdots < T_n$. In particular, there is no functional convergence in the space of continuous processes.
\end{example}

\begin{example}[Rank-based diffusions] \label{ex_2}
Consider the $N$-particle system evolving according to
\[
dX^{i,N}_t = B\left( F^N_t( X^{i,N}_t) \right) dt + \sqrt{2} dW^i_t
\]
for $i \in [N]$, where $B(r)$ is a twice continuously differentiable function on $[0,1]$ and
\[
F^N_t(x) = \frac{1}{N} \sum_{j=1}^N \bm1_{\{X^{j,N}_t \le x\}}
\]
is the empirical distribution function. Such systems are called \emph{rank-based} because the drift (and in more general formulations also the diffusion) of each particle depends on its \emph{rank} within the population. Indeed, module tie-breaking, $F^N_t( X^{i,N}_t) = k/N$ where $k=1$ if $X^{i,N}_t$ is the smallest particle, $k=2$ if $X^{i,N}_t$ is the second smallest, and so on. Rank-based systems have been studied extensively and play an important role in stochastic portfolio theory; see e.g.\ \cite{MR1894767,MR2914770,MR3327514,MR3340241,MR3773380}. They are challenging to analyze in part because the drift is discontinuous as a function of the current state and the empirical measure (with the Wasserstein metric $\Wcal_p$ for any $p \geq 1$), making e.g.\ the argument in Remark~\ref{R_Lacker} inapplicable.

The above system fits into our setup by taking $A(t,\bm x_{[0,t]}) = \sqrt{2}$, $B(t,\bm x_{[0,t]},r) = B(r)$, $C(t, \bm x_{[0,t]}) = 0$, and $g(t,\bm x_{[0,t]},\bm y_{[0,t]}) = \bm1_{\{y_t \le x_t\}}$. Clearly Assumption~\ref{ass1} is satisfied. The limiting McKean--Vlasov equation takes the form
\begin{align*}
dX_t & = B( F_t(X_t) ) dt + \sqrt{2} dW_t, \\
F_t(x) & = \P(X_t \le x).
\end{align*}
The above setup is well-studied, and both the $N$-particle system and the McKean--Vlasov equation are well-posed \cite{MR2914770,MR3327514}. Since the interaction function $g$ and drift coefficient $B$ are both bounded, Assumption~\ref{ass2} is readily seen to be satisfied.

General criteria for verifying the domain of attraction assumption on $X$ are not available. However, if $X$ is stationary, more can be said. It is known \cite{MR2914770,MR3327514} that the distribution function $F_t(x)$ satisfies the PDE
\[
\partial_t F = \partial_{xx} F - \partial_x \mathfrak B(F),
\]
where $\mathfrak B(u) = \int_0^u B(r) dr$. Let us assume that $B(0), B(1) \ne 0$, $\mathfrak B(u) > 0$ for all $u \in (0,1)$, and $\mathfrak B(1) = 0$. In this case there is a solution $F(x)$ to the stationary equation
\begin{equation} \label{eq_F_ODE}
F'' = \mathfrak B(F)'
\end{equation}
which is a distribution function. By using $F$ as initial condition for $X_0$, the solution of the McKean--Vlasov equation has constant marginal law, $\P(X_t \le x) = F(x)$ for all $t$ and $x$. By integrating \eqref{eq_F_ODE} once and using that $F(-\infty) = F'(-\infty) = \mathfrak B(0) = 0$ one obtains
\begin{equation} \label{eq_F_ODE_2}
F' = \mathfrak B(F).
\end{equation}
(Here it becomes clear why $\mathfrak B \ge 0$ and $\mathfrak B(1) = 0$ are needed, as $F'$ is a probability density.) We now apply the von~Mises condition \cite[Theorem~1.1.8]{HF10}, which states that $F$ belongs to the Gumbel domain of attraction if
\[
\lim_{x \to \infty} \frac{(1 - F(x)) F''(x)}{F'(x)^2} = -1.
\]
The mean value theorem yields $\mathfrak B(F(x)) = \mathfrak B(F(x)) - \mathfrak B(1) = B(r^*)(F(x)-1)$ for some $r^* \in (F(x),1)$. Next, \eqref{eq_F_ODE_2} implies that $F'' = B(F)\mathfrak B(F)$. Thus,
\[
\frac{(1 - F(x)) F''(x)}{F'(x)^2} = \frac{(1-F(x))B(F(x))}{\mathfrak B(F(x))} = - \frac{B(F(x))}{B(r^*)} \to -\frac{B(1)}{B(1)} = -1 \text{ as } x \to \infty.
\]
This confirms that the hypotheses of Theorem~\ref{T_main_precise} are satisfied. We deduce that Gumbel asymptotics hold for the $N$-particle systems,
\[
\P\left( \max_{i \le N} \frac{X^{i,N}_T - b^N_T}{a^N_T} \le x \right) \to \exp(-e^{-x}) \text{ as } N \to \infty.
\]
\end{example}

\section{Key lemmas}\label{sect4}

As discussed in Section~\ref{S_proof_outline}, the proof of Theorem~\ref{T_main_precise} relies on counting the nonzero terms of the form \eqref{eq_proof_sketch_4} and bounding their size. This was done under a smallness assumption on $T$ which allows us to truncate the chaos expansion \eqref{eq_chaos} at a finite level. In order to perform this truncation without any smallness assumption on $T$, we have to partition the interval $(0,T]$ into a sufficiently large number $n$ of subintervals $(T_{\alpha-1},T_\alpha]$, $\alpha \in [n]$. Doing so leads to expressions analogous to \eqref{eq_proof_sketch_4} but more complex, and it is those expressions that we need to control. Lemmas~\ref{L_iterated_integrals} and~\ref{count} control the number of nonzero expressions. Lemmas~\ref{tailest} and~\ref{lpestimationiterated} provide tail bounds on iterated stochastic integrals which are used to bound the size of the nonzero expressions and control the error that we commit when truncating the chaos expansions, among other things. The proofs of the lemmas are given in Section~\ref{S_proofs_of_key_lemmas}.

We work with the notation and assumptions of Section~\ref{S_results}. In particular, Assumptions~\ref{ass1} and \ref{ass2} are in force. We also use the notation
\[
\Fcal^V_t = \sigma(X^i_s,W^i_s \colon s \le t, i \in V) \text{ for any index set } V \subset \N,
\]
and write $\mathbb L$ for the space of all progressively measurable processes $Y$ with locally integrable moments, $\int_0^t \E[ |Y_s|^p ] ds < \infty$ for all $t \in \R_+$ and $p \in \N$.

We fix a family of progressively measurable processes $\nG^{\ir\jr} \in \mathbb{L}$, $\ir,\jr \in [N]$, such that $\nG^{\ir \jr}$ is adapted to the filtration $(\Fcal_t^{\{\ir,\jr\}})_{t \ge 0}$, and introduce the iterated integral notation
\begin{equation} \label{iterated_integral_I}
    I_{{\bm i},{\bm j}}^{N}(s, t) = \int_s^t G_{t_1}^{i_1,j_1} \int_s^{t_1} G_{t_2}^{i_2,j_2}\cdots\int_s^{t_{k-1}}G_{t_k}^{i_k,j_k} dW_{t_k}^{i_k}\cdots dW_{t_1}^{i_1}
\end{equation}
for any $k \in \N$ and any multiindices ${\bm i} = (\ir_1,\ldots,\ir_k) \in [N]^k$ and ${\bm j} = (\jr_1,\ldots,\jr_k) \in [N]^k$. Our first key lemma is the following, where later on the random variable $\Psi$ will be instantiated as products of indicators as in \eqref{eq_proof_sketch_4}.

\begin{lemma}[criteria for zero expectation] \label{L_iterated_integrals}
Assume for all $V \subset [N]$, $\ir \in V$, $\jr \notin V$ that
\begin{equation} \label{L_iterated_integrals_eq_0}
\E[ \nG^{\ir\jr}_s \mid \Fcal_t^V ] = 0, \quad s \le t.
\end{equation}
Let $T \ge 0$ and $n \in \N$. Consider an increasing finite sequence $0 = T_0 \le \cdots \le T_n = T$ and fix natural numbers $k_1,\ldots,k_n \in \N$. For each $\alpha \in [n]$, fix two $k_\alpha$-tuples
\[
{\bm i_{\alpha}} = (\ir_{\alpha, 1}, \ldots,\ir_{\alpha, k_\alpha}) \in [N]^{k_\alpha},
\qquad
{\bm j_{\alpha}} = (\jr_{\alpha, 1}, \ldots,\jr_{\alpha, k_\alpha}) \in [N]^{k_\alpha}.
\]
Finally, let $K \subset [N]$ and consider a bounded $\Fcal_T^K$-measurable random variable $\Psi$. Assume at least one of the following conditions is satisfied:
\begin{enumerate}
\item\label{L_iterated_integrals_1} there exist some $\beta \in [n]$ and $\ell_0 \in [k_{\beta}]$ such that
\begin{equation} \label{L_iterated_integrals_eq_1}
\ir_{\beta, \ell_0} \notin K \cup \{\jr_{\beta, 1}, \ldots, \jr_{\beta, \ell_0-1} \} \cup \bigcup_{\alpha = \beta + 1}^n \{\jr_{\alpha, 1},\ldots,\jr_{\alpha, k_\alpha} \},
\end{equation}
where $\{\jr_{\beta, 1}, \ldots, \jr_{\beta, \ell_0-1} \}$ is regarded as the empty set when $\ell_0 = 1$,

\item\label{L_iterated_integrals_2} one has
\[
\jr_{1, k_1} \notin K \cup \{\jr_{1,1}, \ldots, \jr_{1,k_1-1} \} \cup \bigcup_{\alpha = 2}^n \{\jr_{\alpha,1},\ldots,\jr_{\alpha, k_\alpha} \}.
\]
\end{enumerate}
Then
\begin{equation} \label{eq_L_iterated_integrals_2}
\E\left[ \Psi \prod_{\alpha \in [n]} I_{{\bm i_{\alpha}},{\bm j_{\alpha}}}^{N}(T_{\alpha - 1} , T_{\alpha}) \right] = 0.
\end{equation}
\end{lemma}

The criteria \ref{L_iterated_integrals_1} and \ref{L_iterated_integrals_2} in Lemma~\ref{L_iterated_integrals} for zero expectation are of a combinatorial nature involving index set membership. The following lemma counts the number ways in which these conditions can fail, thereby bounding the number of nonzero terms.

\begin{lemma}[counting lemma]\label{count} 
Fix natural numbers $n, N, \kappa, k_1,\ldots,k_n$. The number of ways we can pick a subset $K \subset [N]$ with $|K| = \kappa$ along with tuples $\ib_{\alpha}, \jb_{\alpha} \in [N]^{k_{\alpha}}$ for all $\alpha \in [n]$ such that both properties \ref{L_iterated_integrals_1} and \ref{L_iterated_integrals_2} of Lemma~\ref{L_iterated_integrals} fail to hold is bounded by
\[
\binom{N}{\kappa} \kappa (\kappa+1) \cdots (\kappa+S) N^{S-1},
\]
where $S = k_1 + \cdots + k_n$.
\end{lemma}

We next develop bounds on iterated stochastic integrals. The following lemma will allow us to truncate the chaos expansions of ratios $Z^N_t / Z^N_s$ at levels that do not need to increase with $N$ to comply with given error tolerances. Note that the lemma gives an upper bound that is summable in $m$ only if $t-s$ is sufficiently small. This is the reason we are forced to partition $[0,T]$ into subintervals when proving Theorem~\ref{T_main_precise} without any smallness assumption on $T$. 

\begin{lemma}[first iterated integral $L^p$ estimate]\label{tailest}
Let
\[
I_m^N(s,t) = \int_s^t\int_s^{t_1}\cdots\int_s^{t_{m-1}}dM_{t_m}^{N}dM_{t_{m-1}}^{N}\cdots dM_{t_1}^{N}
\]
where $M^N$ is defined in \eqref{eq_MN} and it is understood that $I^N_1(s,t) = M^N_t - M^N_s$. Then, for any $N,m, p \in \mathbb{N}$, any $T \in (0,\infty)$, and all $s, t \in \left[0, \, T\right]$ we have
\[
\| I_m^N(s,t)  \|_{2p} \leq \left(C(T)p\sqrt{t - s}\right)^m
\]
where the constant $C(T)$ only depends on $T$ and the bounds from Assumptions~\ref{ass1}--\ref{ass2}.
\end{lemma}

\begin{remark}
Note that we only consider $L^{2p}$ norms for positive integers $p$, which is all that is needed later on. This is why Assumption~\ref{ass2} only involves even integer moments.
\end{remark}

The proof of Lemma~\ref{tailest} relies on the following sharp iterated integral estimate, valid for any continuous local martingale $M$ and any $p \in [1,\infty)$, which follows from \cite[Theorem~1]{CAKR91} on noting that $1 + \sqrt{1 + 1/(2p)} < 3$ for any such $p$:
\begin{equation}\label{sharpestiapplied}
\left\Vert \int_s^t\int_s^{t_1}\cdots\int_s^{t_{m-1}}dM_{t_m} dM_{t_{m-1}} \cdots dM_{t_1} \right\Vert_{2p} \leq \frac{(2pm)^{m/2}3^m}{m!} \left\Vert \langle M \rangle_{s,t}^{1/2} \right\Vert_{2pm}^{m},
\end{equation}
where we write $\langle M \rangle_{s,t} = \langle M \rangle_t - \langle M \rangle_s$ for brevity.

While \eqref{sharpestiapplied} is instrumental for proving Lemma~\ref{tailest}, it cannot be used to bound the iterated integrals appearing in \eqref{eq_L_iterated_integrals_2}, which involve several different local martingales. In order to control the nonzero terms of the form \eqref{eq_L_iterated_integrals_2} we will instead use a weaker estimate obtained by repeated application of the BDG and H\"older inequalities. Fortunately this is sufficient thanks to the sharp control on the number of nonzero terms afforded by Lemmas~\ref{L_iterated_integrals} and~\ref{count}. The following general estimate for iterated stochastic integrals involving several continuous local martingales serves this purpose, and it is also used in the proof of Lemma~\ref{L_iterated_integrals} as well as to control linearization errors when reducing from general drift coefficients $B$ to linear ones in the proof of the main result.

\begin{lemma}[second iterated integral $L^p$ estimate] \label{lpestimationiterated}
For any set of $k \in \N$ continuous local martingales $M^1,\ldots,M^k$ and any $p \in (1,\infty)$ we have the estimate
\[
\left\Vert\int_s^t\int_s^{t_1}\cdots\int_s^{t_{k-1}}dM_{t_k}^k\cdots dM_{t_1}^1\right\Vert_p\leq (4\sqrt{p})^{k} 2^{k(k-1)/4}\prod_{\ell=1}^{k}\|\langle M^{\ell} \rangle_{s,t}^{1/2}\|_{2^{\ell}p} .
\]
\end{lemma}

We end this section with an algebraic estimate which will allow us to combine Lemma~\ref{count} and Lemma~\ref{lpestimationiterated} to show that \eqref{eq_proof_sketch_2} indeed tends to zero as $N \to \infty$.

\begin{lemma}\label{combinatoryinequality} For any $C \in (1,\infty)$, $N, S \in \N$, one has the inequality
\[
\sum_{\kappa=1}^N\binom{N}{\kappa}\kappa(\kappa+1)\cdots(\kappa+S)\left(\frac{C}{N}\right)^{\kappa(1 - 1/\log N)}\leq (S+2)(S+1)^{2(S+1)}e^{Ce}(Ce)^{S+1}.
\]
\end{lemma}

\section{Proofs of the key lemmas} \label{S_proofs_of_key_lemmas}

In this section we prove the lemmas presented in Section~\ref{sect4}. We start with the proof of Lemma~\ref{lpestimationiterated} because it is used in the proof of Lemma~\ref{L_iterated_integrals}.

\subsection{Proof of Lemma~\ref{lpestimationiterated}}

We prove the lemma by induction. The base case $k=1$ follows from the sharp BDG inequality $(7)$ in \cite{CAKR91} and H\"older's inequality. For the induction step, we assume that the inequality holds for any $p>1$ with $k$ replaced by $k-1$. Applying the sharp BDG inequality, H\"older's inequality, Doob's maximal inequality (e.g.\ Theorem 5.1.3 in \cite{COEL15} with $p$ replaced by $2p \geq 2$ and so $q \leq 2$), and finally the induction hypothesis yield
\begin{align*}
    \bigg\Vert\int_s^t\int_s^{t_1}&\cdots\int_s^{t_{k-1}}dM_{t_k}^k\cdots dM_{t_1}^1\bigg\Vert_p \\
    &\leq
    2\sqrt{p}\left\Vert\left(\int_s^t\left(\int_s^{t_1}\cdots\int_s^{t_{k-1}}dM_{t_k}^k\cdots dM_{t_2}^2\right)^2d\langle M^1\rangle_{t_1}\right)^\frac{1}{2}\right\Vert_p\\
    &\leq 2\sqrt{p}\left\Vert\sup_{s\leq t_1\leq t}\bigg|\int_s^{t_1}\cdots\int_s^{t_{k-1}}dM_{t_k}^k\cdots dM_{t_2}^2\bigg|\, \langle M^1\rangle_{s,t}^{1/2}\right\Vert_p\\
    &\leq 2\sqrt{p}\left\Vert\sup_{s\leq t_1\leq t}\bigg|\int_s^{t_1}\cdots\int_s^{t_{k-1}}dM_{t_k}^k\cdots dM_{t_2}^2\bigg|\right\Vert_{2p}\, \left\Vert \langle M^1\rangle_{s,t}^{1/2}\right\Vert_{2p}\\
    &\leq 2\sqrt{p}\, 2\left\Vert\int_s^{t}\cdots\int_s^{t_{k-1}}dM_{t_k}^k\cdots dM_{t_2}^2\right\Vert_{2p}\, \left\Vert \langle M^1\rangle_{s,t}^{1/2} \right\Vert_{2p}\\
    &\leq 2\sqrt{p}\, 2\, (4\sqrt{2p} )^{k-1} 2^{(k-1)(k-2)/4}\prod_{\ell=1}^{k-1}\left\Vert \langle M^{\ell+1}\rangle_{s,t}^{1/2}\right\Vert_{2^{\ell+1}p} \left\Vert \langle M^1\rangle_{s,t}^{1/2} \right\Vert_{2p}\\
    &=(4\sqrt{p})^{k} 2^{k(k-1)/4}\prod_{\ell=1}^{k}\| \langle M^{\ell} \rangle_{s,t}^{1/2}\|_{2^{\ell}p}.
\end{align*}

\subsection{Proof of Lemma~\ref{L_iterated_integrals}}

We will need the following two auxiliary lemmas on conditioning, the proofs of which are a trivial modification of the proof of Lemma 2.1.4 in \cite{LEDThesis}.

\begin{lemma} \label{L_cond_exp_1}
For any Brownian motion $W$, two processes $a, b \in \mathbb{L}$, and a $\sigma$-algebra $\Gcal$ such that $a(s)$ and $W(s)$ are $\Gcal$-measurable for $s \le t$, one has
\[
\E\left[ \int_0^t a(s) b(s) dW(s) \Mid \Gcal \right] = \int_0^t a(s) \E[ b(s) \mid \Gcal ]  dW(s).
\]
\end{lemma}

\begin{lemma} \label{L_cond_exp_2}
For any Brownian motion $W$, a processs $a \in \mathbb{L}$, and a $\sigma$-algebra $\Gcal$ such that $W$ is independent of $\Gcal$, one has
\[
\E\left[ \int_s^t a(u) dW(u) \Mid \Fcal_s \vee \Gcal \right] = 0, \quad s \le t.
\]
\end{lemma} 

Notice that these lemmas can be applied for $a$ and $b$ being either the processes $G^{ij}$, which belong to $\mathbb{L}$ by definition, or the iterated integrals $I_{{\bm i},{\bm j}}^{N}(s, \, t)$ for ${\bm i} = (\ir_1,\ldots,\ir_k) \in [N]^k$ and ${\bm j} = (\jr_1,\ldots,\jr_k) \in [N]^k$, which also belong to $\mathbb{L}$. (Recall that these iterated integrals are defined in \eqref{iterated_integral_I}.) The latter can be seen by applying Lemma~\ref{lpestimationiterated} with $M^{\ell}_t = \int_{0}^{t}G_t^{i_{\ell}j_{\ell}}dW_s^{i_{\ell}}$ for all $\ell \in [k]$ and then H\"older's inequality. 

Assume now that condition \ref{L_iterated_integrals_1} is satisfied. Let $\beta \in [n]$ be the largest index such that \eqref{L_iterated_integrals_eq_1} holds for some $\ell_0 \in [k_{\beta}]$, and then let $\ell_0$ be the smallest index for which this happens. Now define
\[
V = K \cup \{\jr_{1,1}, \ldots, \jr_{1,k_1-1} \} \cup \bigcup_{\alpha = 2}^n \{\jr_{\alpha,1},\ldots,\jr_{\alpha, k_\alpha} \}.
\]
Maximality of $\beta$ implies that $\ir_{\alpha, \ell} \in V$ for all $\alpha \ge \beta + 1$ and all $\ell \in [k_\alpha]$. Moreover, by definition of $V$ we have $j_{\alpha, \ell} \in V$ for all $\alpha \ge \beta + 1$ and all $\ell \in [k_\alpha]$. Thus every index appearing in $\ib_\alpha$ or $\jb_\alpha$ for $\alpha \ge \beta + 1$ belongs to $V$. As a result, $I_{{\bm i_{\alpha}},{\bm j_{\alpha}}}^{N}(T_{\alpha - 1} , T_{\alpha})$ is $\Fcal_T^V$-measurable for all $\alpha \ge \beta + 1$. Since $K \subset V$, $\Psi$ is also $\Fcal_T^V$-measurable. Finally, for $\alpha \le \beta - 1$ we have that $I_{{\bm i_{\alpha}},{\bm j_{\alpha}}}^{N}(T_{\alpha - 1} , T_{\alpha})$ is $\Fcal_{T_{\beta-1}}$-measurable. We conclude that
\begin{align*}
& \E\left[ \Psi \prod_{\alpha \in [n]}I_{{\bm i_{\alpha}},{\bm j_{\alpha}}}^{N}(T_{\alpha - 1} , T_{\alpha}) \right] \\
&\qquad = \E\left[ \Psi \prod_{\substack{\alpha \in [n] \\ \alpha \ne \beta}} I_{{\bm i_{\alpha}},{\bm j_{\alpha}}}^{N}(T_{\alpha - 1} , T_{\alpha}) \E\left[ I_{{\bm i_{\beta}},{\bm j_{\beta}}}^{N}(T_{\beta - 1} , T_{\beta}) \Mid \Fcal_{T_{\beta-1}} \vee \Fcal_T^V \right]  \right].
\end{align*}
It remains to show that
\begin{equation} \label{L_iterated_integrals_eq_2}
\E\left[ I_{{\bm i_{\beta}},{\bm j_{\beta}}}^{N}(T_{\beta - 1} , T_{\beta}) \Mid \Fcal_{T_{\beta-1}} \vee \Fcal_T^V \right] = 0,
\end{equation}
and this will rely on repeated application of Lemma~\ref{L_cond_exp_1}. Note that $j_{\beta, \ell} \in V$ for all $\ell \le \ell_0 - 1$ by definition of $V$. Moreover, minimality of $\ell_0$ implies that $\ir_{\beta, \ell} \in V$ for all $\ell \le \ell_0-1$. For $\ell$ in this range, starting with $\ell = 1$, we may therefore apply Lemma~\ref{L_cond_exp_1} iteratively with
\begin{align*}
\Gcal &= \Fcal_{T_{\beta-1}} \vee \Fcal_T^V, \\
W(t) &= W^{\ir_{\beta, \ell}}_{t}, \\
a(t) &= G^{\ir_{\beta, \ell} \jr_{\beta, \ell}}_{t}, \\
b(t) &= I_{{\bm i_{\beta}^{(\ell+1)}},{\bm j_{\beta}^{(\ell+1)}}}^{N}(T_{\beta - 1} , t),
\end{align*}
where $\ib_\beta^{(\ell)} = (\ir_{\beta, \ell}, \ldots, \ir_{\beta, k_{\beta}})$ and $\jb_\beta^{(\ell)} = (\jr_{\beta, \ell}, \ldots, \jr_{\beta, k_{\beta}}),$ to obtain
\begin{align*}
&\E\left[ I_{{\bm i_{\beta}},{\bm j_{\beta}}}^{N}(T_{\beta - 1} , T_{\beta}) \Mid \Fcal_{T_{\beta-1}} \vee \Fcal_T^V \right] \\
& \qquad = \E\left[ \int_{T_{\beta-1}}^{T_\beta} G^{\ir_{\beta, 1} \jr_{\beta, 1}}_{t_1} I_{{\bm i_{\beta}^{(2)}},{\bm j_{\beta}^{(2)}}}^{N}(T_{\beta - 1} , t_1) dW^{i_{\beta, 1}}_{t_1} \Mid \Fcal_{T_{\beta-1}} \vee \Fcal_T^V \right] \\
& \qquad = \int_{T_{\beta-1}}^{T_\beta} G^{\ir_{\beta, 1} \jr_{\beta, 1}}_{t_1} \E\left[ I_{{\bm i_{\beta}^{(2)}},{\bm j_{\beta}^{(2)}}}^{N}(T_{\beta - 1} , t_1) \Mid \Fcal_{T_{\beta-1}} \vee \Fcal_T^V \right]  dW^{i_{\beta, 1}}_{t_1} \\
& \qquad \ \vdots \\
& \qquad = \int_{T_{\beta-1}}^{T_{\beta}} \nG^{\ir_{\beta, 1} \jr_{\beta, 1}}_{t_1} \cdots \int_{T_{\beta-1}}^{t_{\ell_0 - 2}} \nG^{\ir_{\beta, \ell_0-1} \jr_{\beta, \ell_0-1}}_{t_{\ell_0-1}} \\
& \qquad \qquad \qquad \E\left[ I_{{\bm i_{\beta}^{(\ell_0)}},{\bm j_{\beta}^{(\ell_0)}}}^{N}(T_{\beta - 1} , t_{\ell_0 - 1}) \Mid \Fcal_{T_{\beta-1}} \vee \Fcal_T^V \right] dW^{\ir_{\beta, \ell_0-1}}_{t_{\ell_0-1}} \cdots dW^{\ir_{\beta, 1}}_{ t_{1}}.
\end{align*}
The right-hand side of the last is zero. Indeed, we have
\[
 I_{{\bm i_{\beta}^{(\ell_0)}},{\bm j_{\beta}^{(\ell_0)}}}^{N}(T_{\beta - 1} , t) = \int_{T_{\beta - 1}}^{t}  G^{i_{\beta, \ell_0}, j_{\beta, \ell_0}}_{t_{\ell_0}} I_{{\bm i_{\beta}^{(\ell_0 + 1)}},{\bm j_{\beta}^{(\ell_0 + 1)}}}^{N}(T_{\beta - 1} , t_{\ell_0}) dW^{\ir_{\beta, \ell_0}}_{t_{\ell_0}}, \quad t \ge T_{\beta - 1},
\]
where $W^{\ir_{\beta, \ell_0}}$ is independent of $\Fcal_T^V$ because $i_{\beta, \ell_0} \notin V$ due to \eqref{L_iterated_integrals_eq_1}. We therefore deduce from Lemma~\ref{L_cond_exp_2} that
\[
\E\left[ I_{{\bm i_{\beta}^{(\ell_0)}},{\bm j_{\beta}^{(\ell_0)}}}^{N}(T_{\beta - 1} , t) \Mid \Fcal_{T_{\beta-1}} \vee \Fcal_T^V \right] = 0, \quad t \ge T_{\beta - 1}.
\]
This yields \eqref{L_iterated_integrals_eq_2} as required.

Next, assume that condition \ref{L_iterated_integrals_2} is satisfied. In addition, we may assume that condition \ref{L_iterated_integrals_1} does \emph{not} hold since otherwise we would fall in the case just treated. We then define
\[
V = K \cup \{\jr_{1,1}, \ldots, \jr_{1,k_1-1} \} \cup \bigcup_{\alpha = 2}^n \{\jr_{\alpha,1},\ldots,\jr_{\alpha, k_\alpha} \}
\]
and observe that $\ir_{\alpha, \ell} \in V$ for all $\alpha \in [n]$ and all $\ell \in [k_\alpha]$ (since condition \ref{L_iterated_integrals_1} does not hold), and that $\jr_{\alpha, \ell} \in V$ for all $\alpha \in [n]$ and all $\ell \in [k_\alpha]$ except if $(\alpha, \ell) = (1, k_1)$ (by definition of $V$ and since \ref{L_iterated_integrals_2} holds). In particular, $I_{{\bm i_{\alpha}},{\bm j_{\alpha}}}^{N}(T_{\alpha - 1} , T_{\alpha})$ is $\Fcal_T^V$-measurable for all $\alpha \ge 2$, and as before $\Psi$ is $\Fcal_T^V$-measurable as well. Thus
\begin{align*}
& \E\left[ \Psi \prod_{\alpha \in [n]} I_{{\bm i_{\alpha}},{\bm j_{\alpha}}}^{N}(T_{\alpha - 1} , T_{\alpha}) \right] \\
&\qquad = \E\left[ \Psi \prod_{\alpha = 2}^n I_{{\bm i_{\alpha}},{\bm j_{\alpha}}}^{N}(T_{\alpha - 1} , T_{\alpha}) \E\left[ I_{{\bm i_1},{\bm j_1}}^{N}(T_{0} , T_{1}) \Mid \Fcal_T^V \right]  \right].
\end{align*}
The same iterative application of Lemma~\ref{L_cond_exp_1} as before, but now with $\Gcal = \Fcal_T^V$ and using that $W^{\ir_{1,\ell}}_t$, $t \le T$, is $\Fcal_T^V$-measurable for all $\ell \in [k_1]$ and that $\nG^{\ir_{1,\ell}\jr_{1,\ell}}_{t}$, $t \le T$, is $\Fcal_T^V$-measurable for all $\ell \in [k_1-1]$, leads to
\begin{align*}
&\E\left[ I_{{\bm i_1},{\bm j_1}}^{N}(T_{0} , T_{1}) \Mid \Fcal_T^V \right] \\
& \qquad = \int_{T_{0}}^{T_{1}} \nG^{\ir_{1, 1} \jr_{1, 1}}_{t_1} \cdots \int_{T_{0}}^{t_{k_1 - 1}} \E\left[ \nG^{\ir_{1, k_1} \jr_{1, k_1}}_{t_{k_1}} \Mid \Fcal_T^V \right] dW^{\ir_{1, k_1}}_{ t_{k_1} } \cdots dW^{\ir_{1, 1}}_{ t_{1} }.
\end{align*}
The conditional expectation on the right-hand side is equal to zero for all $t_{k_1} \le T$, thanks to \eqref{L_iterated_integrals_eq_0} and the fact that $i_{1, k_1} \in V$ and $j_{1, k_1} \notin V$. This completes the proof of the lemma.

\subsection{Proof of Lemma~\ref{count}}

First we pick the subset $K = \{i_{0,1}, \ldots, i_{0,\kappa}\} \subset [N]$, which can be done in exactly $\binom{N}{\kappa}$ ways. Next, we pick the coordinates of the vectors ${\bm j_{1}} = (\jr_{1, 1}, \ldots,\jr_{1, k_1}), \ldots, {\bm j_{n}} = (\jr_{n, 1}, \ldots,\jr_{n, k_n})$. There are $N$ possible choices for the first $k_1 - 1$ coordinates $\jr_{1, 1}, \ldots,\jr_{1, k_1-1}$ of ${\bm j_{1}}$, and also $N$ choices for all the $k_\alpha$ coordinates of ${\bm j_\alpha}$ for $\alpha \in \{2, \ldots, n\}$. Therefore, we can pick all these coordinates in $N^{k_1 - 1 + k_2 + \ldots + k_n} = N^{S - 1}$ ways. Then, the $k_1$-th coordinate $j_{1,k_1}$ of ${\bm j_1}$ needs to be taken equal to either one of the other $k_1 - 1 + k_2 + k_3 + \cdots + k_n = S - 1$ coordinates we have already picked or one of the $\kappa$ elements of $K = \{i_{0,1}, \ldots, i_{0,\kappa}\}$, as otherwise the second condition \ref{L_iterated_integrals_2} of Lemma~\ref{L_iterated_integrals} will be satisfied. The latter can be done in at most $S - 1 + \kappa$ ways. Hence, there are at most $\binom{N}{\kappa}(\kappa + S - 1)N^{S - 1}$ ways to pick the subset $K$ of $[N]$ and all the coordinates of ${\bm j_1}, \ldots, {\bm j_n}$. Finally, we pick the coordinates of the vectors ${\bm i_1},\ldots, {\bm i_n}$, where for each $\beta \in [n]$ and $\ell \in [k_\beta]$ we must take 
\[
\ir_{\beta, \ell} \in K \cup \{\jr_{\beta, 1}, \ldots, \jr_{\beta, \ell-1} \} \cup \bigcup_{\alpha = \beta + 1}^n \{\jr_{\alpha, 1},\ldots,\jr_{\alpha, k_\alpha} \},
\]
so that the first condition \ref{L_iterated_integrals_1} of Lemma~\ref{L_iterated_integrals} will fail to hold. This can be done in at most $u(\beta, \ell) = \kappa + \ell - 1 + k_{\beta+1} + k_{\beta+2} + \ldots + k_{n}$ ways. Observing that $u(\beta, \ell)$ takes every integer value between $u(n, 1) = \kappa$ and $u(1, k_1) = \kappa + S - 1$ exactly once, we see that the coordinates of ${\bm i_1}, \ldots, {\bm i_n}$ can be picked in at most $\kappa(\kappa + 1) \cdots (\kappa + S - 1)$ ways. Therefore, the number of ways in which the entire selection of the elements of $K$ and the coordinates of ${\bm i_1}, \ldots, {\bm i_n}$ and ${\bm j_1}, \ldots, {\bm j_n}$ can be done is at most
\[
\binom{N}{\kappa}(\kappa + S - 1)N^{S - 1} \kappa(\kappa + 1) \cdots (\kappa + S - 1) < \binom{N}{\kappa}\kappa(\kappa + 1)\cdots(\kappa + S)N^{S - 1}.
\]

\subsection{Proof of Lemma~\ref{tailest}}

For $m = 0$ the inequality holds trivially so we can assume that $m \geq 1$. Applying \eqref{sharpestiapplied} with $M = M^N$ yields
\begin{equation} \label{sharpestiapplied_N}
\left\Vert \int_{s}^t\int_{s}^{t_1} \cdots \int_{s}^{t_{m-1}}dM_{t_{m}}^{N} \cdots dM_{t_1}^{N} \right\Vert_{2p} \leq \frac{(2pm)^{m/2}3^m}{m!} \Vert \langle M^{N} \rangle_{s,t}^{1/2} \Vert_{2pm}^{m}. \\
\end{equation}
Next, H\"older's inequality and the fact that the distribution of $\Delta B^{i,N}$ is the same for all $i$ since our system is exchangeable give
\begin{equation} \label{holder1}
\begin{aligned}
\left\Vert \langle M^{N} \rangle_{s,t}^{1/2} \right\Vert_{2pm}^{m} &=  \mathbb{E} \left[\left(\sum_{i=1}^{N}\int_{s}^{t}(\Delta B_u^{i,N})^2du\right)^{pm}\right]^\frac{1}{2p}  \\
&\leq \mathbb{E} \left[(N(t-s))^{pm - 1}\sum_{i=1}^{N} \int_{s}^{t}(\Delta B_u^{i,N})^{2pm}du\right]^\frac{1}{2p}  \\
&= \left(N^{pm}(t - s)^{pm - 1}\int_{s}^{t}\mathbb{E}\left[ (\Delta B_u^{1,N} )^{2pm}\right]du\right)^\frac{1}{2p}.  \\
\end{aligned}
\end{equation}
Letting $C$ be the upper bound on the derivative of $r \mapsto B(t, \bm x_{[0,t]}, r)$ afforded by Assumption~\ref{ass2}, we obtain
\[
\left|\Delta B_u^{1,N}\right|
\leq C\left|\int g(u,X_{[0,u]}^1, \bm{y}) (\mu_u^{N} - \mu_u)(d\bm{y}) \right|.
\]
Using this in \eqref{holder1} yields
\[
\left\Vert \langle M^{N} \rangle_{s, t}^{1/2} \right\Vert_{2pm}^{m}
\leq \Bigg(N^{pm}(t - s)^{pm - 1}C^{2pm}\int_{s}^{t} \mathbb{E}\Bigg[ \left|\int g(u,X_{[0,u]}^1, \bm{y}) (\mu_u^{N} - \mu_u)(d\bm{y}) \right|^{2pm}\Bigg]du \Bigg)^\frac{1}{2p}.
\]
We may now use Assumption~\ref{ass2} to deduce that for any integer $p > 1$,
\begin{align*}
\left\Vert \langle M^{N} \rangle_{s, t}^{1/2} \right\Vert_{2pm}^{m} &\leq \left(N^{pm}(t - s)^{pm - 1}C^{2pm}\int_{s}^{t}\frac{1}{N^{pm}} \left(pm\right)!  K(u)^{pm} du \right)^\frac{1}{2p} \\
&\le (t - s)^{\frac{m}{2}}C^{m}\left(\left(pm\right)!\right)^{\frac{1}{2p}}\sup_{0 \leq u \leq T}K(u)^{m/2}.
\end{align*}
Plugging this into \eqref{sharpestiapplied_N} we obtain
\begin{align}\label{bbb}
&\left\Vert \int_{s}^t\int_{s}^{t_1} \cdots \int_{s}^{t_{m-1}}dM_{t_{m}}^{N} \cdots dM_{t_1}^{N} \right\Vert_{2p} \nonumber \\
&\qquad \leq (3C)^{m}\frac{((pm)!)^{1/2p}}{m!}(2pm)^{\frac{m}{2}}(t - s)^{\frac{m}{2}}\sup_{0 \leq u \leq T}K(u)^{m/2}.
\end{align}
A calculation using Stirling's approximation yields $((pm)!)^{1/2p} / m! < (2pm )^{-m/2}(8e^2p)^{m}$, and substituting this into \eqref{bbb} finally leads to
\[
\Bigg\Vert \int_{s}^t\int_{s}^{t^1} \cdots \int_{s}^{t^{m-1}}dM_{t^{m}}^{N}\cdots dM_{t^1}^{N}  \Bigg\Vert_{2p} \le (3C)^{m}(t - s)^{\frac{m}{2}}(8e^2p)^{m} \sup_{0 \leq u \leq T}K(u)^{m/2}.
\]
This shows that the desired result holds with $C(T) = 24Ce^2\sqrt{\sup_{0 \leq u \leq T}K(u)}$.

\subsection{Proof of Lemma~\ref{combinatoryinequality}}

We recall the identity
\begin{equation} \label{intermediatelemma4.6_0}
x \frac{d^{S+1}}{dx^{S+1}}\sum_{\kappa=0}^N\binom{N}{\kappa}x^{\kappa+S} = \sum_{\kappa=1}^{N}\binom{N}{\kappa}\kappa(\kappa+1) \cdots (\kappa+S-1)(\kappa + S)x^{\kappa}.
\end{equation}
By the binomial theorem and Leibniz's rule for the derivative of a product of functions, we have the estimate
\begin{align}\label{intermediatelemma4.6}
x \frac{d^{S+1}}{dx^{S+1}}\sum_{\kappa=0}^N\binom{N}{\kappa}x^{\kappa+S} &= x \frac{d^{S+1}}{dx^{S+1}} \Bigg(x^S(1+x)^N\Bigg)\nonumber\\
&=x \sum_{i=0}^{S+1}\binom{S+1}{i}S(S-1)\cdots(S-i+1)x^{S-i}\nonumber\\
&\qquad \qquad \times N(N-1)\cdots(N-S+i)(1+x)^{N-S-1+i}\nonumber\\
&\leq (S+1)^{2(S+1)}(1+x)^N\sum_{i=0}^{S+1}(xN)^{S-i+1}.
\end{align}
Combining \eqref{intermediatelemma4.6_0} and \eqref{intermediatelemma4.6}, plugging in $x = (C/N)^{1-1/\log N}$, noting that this value of $x$ is upper bounded by $Ce / N$ since $C>1$, and finally using that $1 + Ce/N \leq \exp(Ce/N)$ and that $\left(eC\right)^{S - i + 1} < \left(eC\right)^{S + 1}$ since $eC > e > 1$, we obtain the desired inequality.

\section{Proof of Theorem~\ref{T_main_precise}} \label{S_main_proof}

We now prove Theorem~\ref{T_main_precise}. The setup of Section~\ref{S_results} will be used. In particular the objects $\mu^N$, $M^N$, $\Delta B^{i,N}$ in \eqref{eq_muN}--\eqref{eq_DeltaBN} as well as the density process $Z^N = \exp(M^N - \frac12 \langle M^N \rangle)$ will be referred to freely. Assumptions~\ref{ass1} and \ref{ass2} are in force. The induced measure $\Q^N$, the normalizing constants $a^N_T,b^N_T$, and the limiting distribution function $\Gamma_T$ are as in the statement of the theorem. The time point $T$ is fixed throughout.

We must prove \eqref{eq_T_main_precise_1}. It suffices to do this for $x \in \R$ such that $\Gamma_T(x) > 0$. Indeed, suppose this has been done and consider $x$ such that $\Gamma_T(x) = 0$. Because all extreme value distributions are continuous, for any $\varepsilon > 0$ there is $x' > x$ such that $0 < \Gamma_T(x') < \varepsilon$, and thus
\[
\Q^N\left( \max_{i \le N} \frac{X^i_T - b^N_T}{a^N_T} \le x \right) \le \Q^N\left( \max_{i \le N} \frac{X^i_T - b^N_T}{a^N_T} \le x' \right) \to \Gamma_T(x') < \varepsilon.
\]
Since $\varepsilon > 0$ was arbitrary, the left-hand side converges to $\Gamma_T(x) = 0$ as $N \to \infty$. We thus pick $x$ such that $\Gamma_T(x) > 0$ and set out to prove that as $N \to \infty$,
\[
\Q^N\left( \max_{i \le N} \frac{X^i_T - b^N_T}{a^N_T} \le x \right) - \P\left( \max_{i \le N} \frac{X^i_T - b^N_T}{a^N_T} \le x \right) = \E\left[ \prod_{i=1}^N \bm1_{\{X^i_T \le x_N\}} (Z^N_T - 1) \right] \to 0,
\]
where for brevity we introduce the notation
\[
x_N = a^N_T x + b^N_T.
\]
The proof is divided into several steps.

\paragraph{Step 1: partitioning the time interval.}

Chaos expansions of $Z^N$ are at the core of the proof, and to get sufficient control on the convergence of these expansions we partition the interval $(0,T]$ into $n$ subintervals $(T_{\alpha-1},T_\alpha]$, $\alpha \in [n]$, of equal length $T_\alpha - T_{\alpha-1} = T/n$. We choose $n$ large enough that $C(T) \sqrt{T/n} < 1/2$, where $C(T)$ is the constant in Lemma~\ref{tailest}, and then keep $n$ fixed for the remained of the proof. We now observe the identity
\[
Z^N_T - 1 = \sum_{\alpha=1}^n \prod_{\beta=1}^\alpha \left( \frac{Z^N_{T_\beta}}{Z^N_{T_{\beta-1}}} - \delta_{\alpha \beta} \right),
\]
where $\delta_{\alpha \beta}$ is the Kronecker delta, thus $\delta_{\alpha \beta}=1$ if $\alpha = \beta$ and $\delta_{\alpha \beta}=0$ otherwise. This yields
\[
\E\left[ \prod_{i=1}^N \bm1_{\{X^i_T \le x_N\}} (Z^N_T - 1) \right] = \sum_{\alpha = 1}^n A^N_\alpha,
\]
where
\[
A^N_\alpha = \E \left[ \prod_{i=1}^N \bm1_{\{X^i_T \le x_N\}} \prod_{\beta=1}^\alpha \left( \frac{Z^N_{T_\beta}}{Z^N_{T_{\beta-1}}} - \delta_{\alpha \beta} \right) \right].
\]
To prove the theorem it suffices to show that $A^N_\alpha \to 0$ as $N \to \infty$ for each $\alpha \in [n]$. We thus fix any such $\alpha$ and set out to prove that $A^N_\alpha \to 0$.

\paragraph{Step 2: controlling the tails of the chaos expansions uniformly in $N$.}

Let $\varepsilon > 0$ be arbitrary. We will show by induction that there are positive integers $m_1, m_2, \ldots, m_{\alpha}$, which do not depend on $N$, such that for $\gamma = 1,\ldots, \alpha + 1$ we have
\begin{equation} \label{chaostailcontrol}
\begin{aligned}
|A^N_{\alpha}| &\leq (\gamma - 1)\varepsilon \\
&\quad + \left|\mathbb{E}\left[\prod_{i=1}^N \bm1_{\{X_{T}^{i} \le x_N\}}\prod_{\beta = 1}^{\gamma - 1}\left(\sum_{m = \delta_{\alpha\beta}}^{m_{\beta}}I_m^N(T_{\beta - 1}, T_\beta)\right) \prod_{\beta = \gamma}^\alpha\left(\frac{Z_{T_\beta}^N}{Z_{T_{\beta - 1}}^N}-\delta_{\alpha\beta}\right)\right]\right|,
\end{aligned}
\end{equation}
with the convention that an empty product is equal to one. The base case $\gamma = 1$ holds trivially because the right-hand side is then just equal to $|A^N_\alpha|$. Suppose now that for some $\gamma \in [\alpha]$ we have determined positive integers $m_1,\ldots,m_{\gamma-1}$ such that \eqref{chaostailcontrol} holds. We will find $m_\gamma$ such that \eqref{chaostailcontrol} is true with $\gamma$ replaced by $\gamma+1$.

To this end, decompose the chaos expansion of $Z^N_{T_\gamma} / Z^N_{T_{\gamma-1}}$ as
\[
\frac{Z_{T_\gamma}^N}{Z_{T_{\gamma - 1}}^N}-\delta_{\alpha\gamma} = \sum_{m=\delta_{\alpha\gamma}}^{m_{\gamma}} I_m^N(T_{\gamma - 1}, T_\gamma) + 
\sum_{m=m_{\gamma} + 1}^{\infty} I_m^N(T_{\gamma - 1}, T_\gamma).
\]
As will become clear shortly, the infinite series converges in $L^2$ thanks to Lemma~\ref{tailest} and the fact that $T_{\gamma}-T_{\gamma-1} = T/n$ is sufficiently small. Plugging this into the induction hypothesis \eqref{chaostailcontrol} we get
\begin{equation} \label{inductionstep}
\begin{aligned}
|A^N_{\alpha}| &\leq (\gamma - 1)\varepsilon \\
&\quad + \Bigg|\mathbb{E}\Bigg[\prod_{i=1}^N \bm1_{\{X_{T}^{i} \le x_N\}}\prod_{\beta = 1}^{\gamma}\Bigg(\sum_{m = \delta_{\alpha\beta}}^{m_{\beta}}I_m^N(T_{\beta - 1}, T_\beta)\Bigg) \prod_{\beta = \gamma+1}^\alpha\Bigg(\frac{Z_{T_\beta}^N}{Z_{T_{\beta - 1}}^N}-\delta_{\alpha\beta}\Bigg)\Bigg]\Bigg| \\
&\quad + \Bigg| \mathbb{E}\Bigg[ \prod_{i=1}^N \bm1_{\{X_{T}^{i} \le x_N\}}\prod_{\beta = 1}^{\gamma - 1}\Bigg(\sum_{m = \delta_{\alpha\beta}}^{m_{\beta}}I_m^N(T_{\beta - 1}, T_\beta)\Bigg) \\
&\qquad\qquad\qquad \times \Bigg(\sum_{m = m_\gamma + 1}^{\infty}I_m^N(T_{\gamma - 1}, T_\gamma)\Bigg) \prod_{\beta = \gamma+1}^\alpha\Bigg(\frac{Z_{T_\beta}^N}{Z_{T_{\beta - 1}}^N}-\delta_{\alpha\beta}\Bigg) \Bigg] \Bigg| .
\end{aligned}
\end{equation}
The third term on the right-hand side of \eqref{inductionstep} is bounded by
\[
\mathbb{E}\Bigg[  \prod_{\beta = 1}^{\gamma - 1} \Bigg| \sum_{m = \delta_{\alpha\beta}}^{m_{\beta}}I_m^N(T_{\beta - 1}, T_\beta) \Bigg| \sum_{m = m_\gamma + 1}^{\infty} \left| I_m^N(T_{\gamma - 1}, T_\gamma) \right|  \E\Bigg[ \Bigg| \prod_{\beta = \gamma+1}^\alpha \Bigg( \frac{Z_{T_\beta}^N}{Z_{T_{\beta - 1}}^N}-\delta_{\alpha\beta} \Bigg) \Bigg| \Mid \Fcal_{T_\gamma} \Bigg] \Bigg].
\]
Note that $| \prod_{\beta = \gamma+1}^\alpha ( Z_{T_\beta}^N / Z_{T_{\beta - 1}}^N - \delta_{\alpha\beta})|$ is bounded by $(Z^N_{T_\alpha} + Z^N_{T_{\alpha-1}}) / Z^N_{T_\gamma}$ if $\gamma \le \alpha-1$, and by one if $\gamma=\alpha$. The martingale property of $Z^N$ thus implies that the conditional expectation above is bounded by two. Using also H\"older's inequality, the triangle inequality, and finally Lemma~\ref{tailest}, we bound the expression in the preceding display by
\begin{align*}
&2\, \mathbb{E}\Bigg[  \prod_{\beta = 1}^{\gamma - 1} \Bigg| \sum_{m = \delta_{\alpha\beta}}^{m_{\beta}}I_m^N(T_{\beta - 1}, T_\beta) \Bigg| \sum_{m = m_\gamma + 1}^{\infty} \left| I_m^N(T_{\gamma - 1}, T_\gamma) \right| \Bigg] \\
&\le 2 \prod_{\beta = 1}^{\gamma - 1} \Bigg\| \sum_{m = \delta_{\alpha\beta}}^{m_{\beta}}I_m^N(T_{\beta - 1}, T_\beta) \Bigg\|_{2\gamma-2} \Bigg\| \sum_{m = m_\gamma + 1}^{\infty} \left| I_m^N(T_{\gamma - 1}, T_\gamma) \right| \Bigg\|_2 \\
&\le 2 \prod_{\beta = 1}^{\gamma - 1} \sum_{m = \delta_{\alpha\beta}}^{m_{\beta}} \left\| I_m^N(T_{\beta - 1}, T_\beta) \right\|_{2\gamma-2}  \sum_{m = m_\gamma + 1}^{\infty} \left\| I_m^N(T_{\gamma - 1}, T_\gamma) \right\|_2 \\
&\le 2 \prod_{\beta = 1}^{\gamma - 1} \sum_{m = \delta_{\alpha\beta}}^{m_{\beta}} \left( C(T) (\gamma-1) \sqrt{T/n} \right)^m \sum_{m = m_\gamma + 1}^{\infty} \left( C(T) \sqrt{T/n} \right)^m.
\end{align*}
Thanks to the choice of $n$ in Step~1, the right-hand side is bounded by
\[
2 \prod_{\beta = 1}^{\gamma - 1} \sum_{m = \delta_{\alpha\beta}}^{m_{\beta}} \left( \frac{\gamma-1}{2} \right)^m \sum_{m = m_\gamma + 1}^{\infty} 2^{-m}.
\]
We now simply choose $m_\gamma$ large enough that this expression is less than $\varepsilon$. Plugging this back into \eqref{inductionstep} yields \eqref{chaostailcontrol} with $\gamma$ replaced by $\gamma+1$. This completes the induction step and shows that \eqref{chaostailcontrol} holds for all $\gamma=1,\ldots,\alpha+1$. In particular, taking $\gamma=\alpha+1$ we obtain
\begin{equation} \label{step1result}
|A^N_{\alpha}| \leq \alpha \varepsilon + \left|\mathbb{E}\left[\prod_{i=1}^N \bm1_{\{X_{T}^{i} \le x_N\}}\prod_{\beta = 1}^{\alpha} \sum_{m = \delta_{\alpha\beta}}^{m_{\beta}}I_m^N(T_{\beta - 1}, T_\beta) \right]\right|.
\end{equation}

\paragraph{Step 3: reduction to linear drift.}

We now linearize the drift function $B(t, \bm x_{[0,t]}, r)$ with respect to its third argument. We write $D_3 B(t, \bm x_{[0,t]}, r)$ for the derivative with respect to $r$ and define for simplicity the process
\[
D_3 B^i_t = D_3 B\left(t, X^i_{[0,t]}, \int g(t, X^i_{[0,t]}, \bm y_{[0,t]}) \mu_t(d\bm y) \right).
\]
Note that $D_3B^i$ is adapted to the filtration $(\Fcal^{\{i\}}_t)_{t \ge 0}$ generated by $(X^i,W^i)$. We also write $H^i_t(\mu) = \int g(t, X^i_{[0,t]}, \bm y_{[0,t]}) \mu(d\bm y)$ for any signed measure $\mu$ on $C(\R_+)$. We then have the Taylor formula
\[
\Delta B^{i,N}_t = D_3 B^i_t \, H^i_t(\mu^N_t - \mu_t) + R^{i,N}_t \left( H^i_t(\mu^N_t - \mu_t) \right)^2,
\]
where $R^{i,N}$ is a process which is uniformly bounded in terms of the bound on the second derivative of $r \mapsto B(t, \bm x_{[0,t]}, r)$ given by Assumption~\ref{ass1}. We now define local martingales
\begin{equation}\label{tildem}
    \widetilde{M}_t^N = \sum_{i=1}^N \int_0^t D_3B_s^i \, H_s^i(\mu_s^N - \mu_s) dW_s^i
\end{equation}
and iterated integrals
\begin{equation}\label{tildei}
    \widetilde{I}_m^N(s,t) = \int_s^t\int_s^{t_1}\cdots\int_s^{t_{m-1}} d\tilde{M}^N_{t_m}\cdots d\tilde{M}^N_{t_1}.
\end{equation}

We will prove that there exists a constant $C$, which does not depend on $N$, such that
\begin{equation}\label{step2claim}
\mathbb{E}\left[ \left| \prod_{\beta = 1}^{\alpha} \sum_{m = \delta_{\alpha\beta}}^{m_{\beta}}I_m^N(T_{\beta - 1}, T_\beta) - \prod_{\beta = 1}^{\alpha} \sum_{m = \delta_{\alpha\beta}}^{m_{\beta}} \widetilde I_m^N(T_{\beta - 1}, T_\beta) \right| \right] \le \frac{C}{\sqrt{N}}.
\end{equation}
To prove this, we expand the products and use the triangle inequality to bound the left hand side by
\[
\sum_{k_1,\ldots,k_\alpha} \mathbb{E} \left[ \left| I_{k_1}^N(T_{0},T_1)\cdots I_{k_\alpha}^N(T_{\alpha - 1},T_\alpha) - \widetilde{I}_{k_1}^N(T_{0},T_1)\cdots \widetilde{I}_{k_\alpha}^N(T_{\alpha - 1},T_\alpha) \right| \right],
\]
where the sum ranges over all $(k_1,\ldots, k_\alpha)$ such that $k_\beta \in [m_{\beta}]\cup \{0\}$ for $\beta < \alpha$ and $k_{\alpha}\in[m_\alpha]$. On each summand in the above expression we apply the identity
\[
\prod_{\beta=1}^\alpha x_\beta - \prod_{\beta=1}^\alpha y_\beta = \sum_{(i_1,\ldots ,i_\alpha)\in\{0,1\}^\alpha \setminus \{\mathbf{0}\}}\prod_{\beta = 1}^\alpha (x_\beta - y_\beta)^{i_\beta}y_\beta^{1-i_\beta}
\]
and then use the triangle inequality along with H\"older's inequality to get
\begin{align} \label{step2eq1}
&\mathbb{E} \left[ \left| I_{k_1}^N(T_{0},T_1)\cdots I_{k_\alpha}^N(T_{\alpha - 1},T_\alpha) - \widetilde{I}_{k_1}^N(T_{0},T_1)\cdots \widetilde{I}_{k_\alpha}^N(T_{\alpha - 1},T_\alpha) \right| \right] \nonumber \\
&\leq
 \sum_{(i_1,\ldots,i_\alpha)\in\{0,1\}^\alpha \setminus \{\mathbf{0}\}} \mathbb{E}\left[ \left| \prod_{\beta = 1}^\alpha \left( I_{k_\beta}^N(T_{\beta - 1}, T_\beta) - \widetilde{I}_{k_\beta}^N(T_{\beta - 1}, T_\beta)\right)^{i_\beta}\widetilde{I}_{k_\beta}^N(T_{\beta - 1}, T_\beta)^{1-i_\beta}\right| \right] \nonumber \\
&\leq
 \sum_{(i_1,\ldots,i_\alpha)\in\{0,1\}^\alpha \setminus \{\mathbf{0}\}}\prod_{\beta = 1}^\alpha \left\| \left(I_{k_\beta}^N(T_{\beta - 1}, T_\beta) - \widetilde{I}_{k_\beta}^N(T_{\beta - 1}, T_\beta) \right)^{i_\beta}\widetilde{I}_{k_\beta}^N(T_{\beta - 1}, T_\beta)^{1-i_\beta} \right\|_{\alpha}.
\end{align}
Thus it suffices to bound each of the products in \eqref{step2eq1} by a constant times $1/\sqrt{N}$. Since each of these products has at least one factor with $i_\beta = 1$, this will follow directly from the estimates
\begin{equation}\label{step2claima}
\|\widetilde{I}_{k_\beta}^N(T_{\beta - 1}, T_\beta)\|_{\alpha} \leq C
\end{equation}
and
\begin{equation}\label{step2claimb}
\|I_{k_\beta}^N(T_{\beta - 1}, T_\beta) - \widetilde{I}_{k_\beta}^N(T_{\beta - 1}, T_\beta)\|_{\alpha} \leq \frac{C}{\sqrt{N}}.
\end{equation}
Once these estimates have been proved, \eqref{step2claim} follows.

To prove \eqref{step2claima}--\eqref{step2claimb} we first derive $L^p$ estimates for the quadratic variations of $\widetilde M^N$ and $M^N - \widetilde M^N$. Let $C$ be a uniform bound on the first and second derivatives of $r \mapsto B(t, \bm x_{[0,t]}, r)$ as given by Assumption~\ref{ass1}, and recall that Assumption~\ref{ass2} gives
\[
\mathbb{E}\left[ \left(H_t^i(\mu_t^N - \mu_t) \right)^{2p} \right] \leq \frac{1}{N^p}p!K(t)^p.
\]
for any positive integer $p$ and $t \in \R_+$. Therefore using H\"older's inequality we obtain, for any positive integer $p$,
\begin{equation} \label{quadratic1}
\begin{aligned}
    \left\| \langle \widetilde{M}^N \rangle_{s,t}^{1/2} \right\|_{2p} &\leq C\, \mathbb{E}\left[\left(\sum_{i=1}^N\int_s^t \left( H_u^i(\mu_u^N - \mu_u) \right)^2  du \right)^{p}\right]^{\frac{1}{2p}} \\
    &\leq C\, (N(t - s))^{\frac{1}{2} - \frac{1}{2p}} \left(\sum_{i=1}^N\int_s^t \mathbb{E}\left[ \left( H_u^i(\mu_u^N - \mu_u) \right)^{2p} \right] du\right)^{\frac{1}{2p}} \\
    &\le C\, (t-s)^{1/2} (p!)^{1/(2p)} \sup_{s \le u \le t} K(u)^{1/2}
\end{aligned}
\end{equation}
and
\begin{equation} \label{quadratic2}
\begin{aligned}
    \left\| \langle M^N - \widetilde{M}^N \rangle_{s,t}^{1/2} \right\|_{2p}
    &\leq C\, \mathbb{E}\left[\left(\sum_{i=1}^N\int_s^t \left( H_u^i(\mu_u^N - \mu_u) \right)^4 du \right)^{p}\right]^{\frac{1}{2p}} \\
    &\leq C\, (N(t - s))^{\frac{1}{2} - \frac{1}{2p}} \left(\sum_{i=1}^N \int_s^t \mathbb{E} \left[ \left(H_u^i(\mu_u^N - \mu_u)\right)^{4p} \right] du \right)^{\frac{1}{2p}} \\
    &\le C\, (t-s)^{1/2} ((2p)!)^{1/(2p)} \sup_{s \le u \le t} K(u) \frac{1}{\sqrt{N}}.
\end{aligned}
\end{equation}
To prove \eqref{step2claima} we apply Lemma~\ref{lpestimationiterated} to the iterated integral $\widetilde{I}_{k_\beta}^N(T_{\beta - 1}, T_\beta)$ and combine this with \eqref{quadratic1} to get
\[
\|\widetilde{I}_{k_\beta}^N(T_{\beta - 1}, T_\beta)\|_{\alpha} \leq C \prod_{\ell=1}^{k_\beta} \left\| \langle \widetilde{M}^N \rangle_{T_{\beta-1},T_{\beta}}^{1/2} \right\|_{2^{\ell}\alpha} \le C.
\]
To prove \eqref{step2claimb} we observe that $I_{k_\beta}^N(T_{\beta - 1}, T_\beta) - \widetilde{I}_{k_\beta}^N(T_{\beta - 1}, T_\beta)$ can be written as the sum of $2^{k_\beta} - 1$ terms, each having the form
\[
\int_{T_{\beta-1}}^{T_\beta}\int_{T_{\beta-1}}^{t_1}\cdots\int_{T_{\beta-1}}^{t_{k_\beta-1}} d{Y}^{k_\beta}_{t_{k_\beta}}\cdots d{Y}^1_{t_1},
\]
where $Y^{\ell} = M^N - \widetilde{M}^N$ for at least one $\ell$ and $Y^{\ell} = \widetilde{M}^N$ for the remaining $\ell$. By first applying Lemma \ref{lpestimationiterated} and then \eqref{quadratic1} and \eqref{quadratic2} we get
\[
\left\| \int_{T_{\beta-1}}^{T_\beta}\int_{T_{\beta-1}}^{t_1}\cdots\int_{T_{\beta-1}}^{t_{k_\beta-1}} d{Y}^{k_\beta}_{t_{k_\beta}}\cdots d{Y}^1_{t_1} \right\|_{\alpha} \leq C \prod_{\ell=1}^{k_\beta}\left\| \langle Y^\ell \rangle_{T_{\beta-1}, T_\beta}^{1/2}\right\|_{2^{\ell}\alpha} \le \frac{C}{\sqrt{N}}.
\]
(Here we used \eqref{quadratic1} for each $Y^{\ell}$ that equals $\tilde{M}^N$ and \eqref{quadratic2} for each $Y^{\ell}$ that equals $M^N - \tilde{M}^N$, and the $1/\sqrt{N}$ factor emerged because there is at least one factor of the latter kind.) By summing and using the triangle inequality we finally obtain \eqref{step2claimb}.

To summarize, we have now proved \eqref{step2claima}--\eqref{step2claimb}, thus showing that each of the products in \eqref{step2eq1} is bounded by a constant times $1/\sqrt{N}$. This in turn yields \eqref{step2claim} as desired.

We end Step~3 by combing \eqref{step2claim} and \eqref{step1result} to get
\begin{equation}\label{AnalyticStep2Claim}
|A^N_{\alpha}| \leq \alpha \varepsilon + \frac{C}{\sqrt{N}} + \left|\mathbb{E}\left[\prod_{i=1}^N \bm1_{\{X_{T}^{i} \le x_N\}}\prod_{\beta = 1}^{\alpha} \sum_{m = \delta_{\alpha\beta}}^{m_{\beta}} \widetilde I_m^N(T_{\beta - 1}, T_\beta) \right]\right|.
\end{equation}
The key point is that the iterated integrals $\widetilde I^N_m(T_{\beta-1},T_\beta)$ are defined in terms of the local martingale $\widetilde M^N$ in \eqref{tildem} which, unlike $M^N$, depends \emph{linearly} on $\mu^N-\mu$. In sense, all nonlinear dependence on $\mu^N-\mu$ has been absorbed into the vanishing term $C/\sqrt{N}$.

\paragraph{Step 4: expanding the iterated integrals.}
Our starting point is now \eqref{AnalyticStep2Claim}, where we recall that $\alpha$ is fixed, $\varepsilon>0$ is arbitrary, and $m_1,\ldots,m_\alpha, C$ do not depend on $N$. Therefore, to show that $A^N_\alpha \to 0$ as $N \to \infty$, it is enough to show that the expectation in \eqref{AnalyticStep2Claim} tends to zero as $N \to \infty$. We now pave the way by expanding the sums and products in \eqref{AnalyticStep2Claim} to bring us into a position where the results of Section~\ref{sect4} can be applied.

We first expand the product indexed by $\beta$ to bound the expectation in \eqref{AnalyticStep2Claim} by
\begin{equation}\label{step2result}
\sum_{k_1,\ldots,k_\alpha} \left|\mathbb{E}\left[\prod_{i=1}^N \bm1_{\{X_{T}^{i} \le x_N\}} \widetilde{I}_{k_1}^N(T_{0},T_1)\cdots \widetilde{I}_{k_\alpha}^N(T_{\alpha - 1},T_\alpha) \right]\right|,
\end{equation}
where the sum ranges over all $(k_1,\ldots, k_\alpha)$ such that $k_\beta \in [m_{\beta}]\cup \{0\}$ for $\beta < \alpha$ and $k_{\alpha}\in[m_\alpha]$. It suffices to show that each summand in \eqref{step2result} vanishes as $N \to \infty$, so we fix a tuple $(k_1,\ldots,k_\alpha)$ and focus on the corresponding expectation.

The next step is to insert the identity $\bm1_{\{X_{T}^{i} \le x_N\}} = 1 - \bm1_{\{X_{T}^{i} > x_N\}}$ and expand the product to write the expectation as
\begin{equation}\label{step3expansion}
\begin{aligned}
&\mathbb{E}\left[\prod_{i=1}^N \bm1_{\{X_{T}^{i} \le x_N\}} \widetilde{I}_{k_1}^N(T_{0},T_1)\cdots \widetilde{I}_{k_\alpha}^N(T_{\alpha - 1},T_\alpha) \right] \\
&\quad =\sum_{\kappa=1}^N (-1)^{\kappa}\sum_{\{i_{01},\ldots,i_{0\kappa}\} \subset [N]}\mathbb{E} \left[ \prod_{\ell=1}^{\kappa} \bm 1_{\{X_{T}^{i_{0\ell}} > x_N\}} \widetilde{I}_{k_1}^N(T_{0},T_1)\cdots \widetilde{I}_{k_\alpha}^N(T_{\alpha - 1},T_\alpha) \right].
\end{aligned}
\end{equation}
The purpose of the substitution $\bm1_{\{X_{T}^{i} \le x_N\}} = 1 - \bm1_{\{X_{T}^{i} > x_N\}}$ is to allow us to use the fact that $\{X_{T}^{i} > x_N\}$ are independent events whose probabilities are of order $1/N$.

We proceed to expand the iterated integrals $\widetilde{I}_{k_{\beta}}^N(T_{\beta - 1}, T_{\beta})$. In view of \eqref{tildem} and the definition of $H_s^i(\mu^N_s - \mu_s)$ and $\mu^N_s$ we have
\begin{align}\label{exhaustiveexpansion}
\widetilde{M}_t^N = \frac{1}{N}\sum_{i=1}^N\sum_{j=1}^N\int_0^tG_s^{ij} dW_s^i,
\end{align}
where
\begin{equation}\label{G}
    G_t^{ij} = D_3B_t^i\, \left(g\left(t,X_{[0,t]}^i,X_{[0,t]}^j\right) - \int g\left(t,X_{[0,t]}^i,\bm y_{[0,t]}\right) \mu_t(d{\bm y}) \right)
\end{equation}
for all $i, j \in [N]$. Plugging \eqref{exhaustiveexpansion} and \eqref{G} into the definition of $\widetilde{I}_{k_{\beta}}^N(T_{\beta - 1}, T_{\beta})$, see \eqref{tildei}, gives
\begin{equation}
    \widetilde{I}^N_{k_\beta}(T_{\beta - 1}, T_\beta) = \frac{1}{N^{k_\beta}}\sum_{{\bm i}\in[N]^{k_\beta}}\sum_{{\bm j}\in[N]^{k_\beta}}I_{{\bm i},{\bm j}}^{N}(T_{\beta - 1}, T_\beta),
\end{equation}
where we use the iterated integral notation \eqref{iterated_integral_I} of Section~\ref{sect4}. The product of iterated integrals appearing in \eqref{step3expansion} can then be written
\[
\widetilde{I}_{k_1}^N(T_{0},T_1)\cdots \widetilde{I}_{k_\alpha}^N(T_{\alpha - 1},T_\alpha)
= \frac{1}{N^{k_1 + \cdots + k_\alpha}}\sum_{({\bm i}_1,{\bm j}_1),\ldots,({\bm i}_\alpha,{\bm j}_\alpha)} \prod_{\beta = 1}^{\alpha} I_{{\bm i_{\beta}},{\bm j_{\beta}}}^{N}(T_{\beta - 1}, T_\beta)
\]
where the sum extends over all $\alpha$-tuples $(({\bm i}_1,{\bm j}_1),\ldots,({\bm i}_\alpha,{\bm j}_\alpha))$ consisting of pairs $({\bm i}_{\beta},{\bm j}_{\beta})$ in $[N]^{k_\beta} \times [N]^{k_\beta}$. Substituting this representation turns the right-hand side of \eqref{step3expansion} into
\[
\frac{1}{N^{k_1 + \cdots + k_\alpha}} \sum_{\kappa=1}^N (-1)^{\kappa}\sum_{\{i_{01},\ldots,i_{0\kappa}\} \subset [N]} \sum_{({\bm i}_1,{\bm j}_1),\ldots,({\bm i}_\alpha,{\bm j}_\alpha)} \mathbb{E} \left[ \prod_{\ell=1}^{\kappa} \bm1_{\{X_{T}^{i_{0\ell}} > x_N\}} \prod_{\beta = 1}^{\alpha} I_{{\bm i_{\beta}},{\bm j_{\beta}}}^{N}(T_{\beta - 1}, T_\beta) \right],
\]
whose absolute value is bounded by
\begin{equation}\label{step3result}
\frac{1}{N^{k_1 + \cdots + k_\alpha}} \sum_{\kappa=1}^N  \sum_{\{i_{01},\ldots,i_{0\kappa}\} \subset [N]} \sum_{({\bm i}_1,{\bm j}_1),\ldots,({\bm i}_\alpha,{\bm j}_\alpha)} \left| \mathbb{E} \left[ \prod_{\ell=1}^{\kappa} \bm1_{\{X_{T}^{i_{0\ell}} > x_N\}} \prod_{\beta = 1}^{\alpha} I_{{\bm i_{\beta}},{\bm j_{\beta}}}^{N}(T_{\beta - 1}, T_\beta) \right] \right|.
\end{equation}
We are now finally in a position where the results of Section~\ref{sect4} can be applied to show that \eqref{step3result} tends to zero as $N\to \infty$. In particular, for small values of $\kappa$, an overwhelming number of expectations in \eqref{step3result} will be zero, while for large values of $\kappa$ we can exploit the smallness of the probabilities $\P(X_{T}^{i_{0\ell}} > x_N)$.

\paragraph{Step 5: application of key lemmas.}
Our focus is on showing that \eqref{step3result} tends to zero as $N \to \infty$, and we recall that $\alpha$ and $k_1,\ldots,k_\alpha$ are fixed.

We first aim to apply Lemma~\ref{L_iterated_integrals} to assert that a large number of the expectations in \eqref{step3result} are in fact zero. We thus fix $\kappa \in [N]$ and instantiate the lemma with $G^{ij}$ as in \eqref{G}, the time points $T_0,\ldots,T_\alpha$, the natural numbers $k_1,\ldots,k_\alpha$, the $k_\beta$-tuples $\bm i_\beta, \bm j_\beta \in [N]^{k_\beta}$ for $\beta \in [\alpha]$, the subset $K = \{i_{01},\ldots,i_{0\kappa}\} \subset [N]$, and the bounded $\Fcal^K_T$-measurable random variable $\Psi = \prod_{\ell=1}^{\kappa} \bm 1_{\{X_{T}^{i_{0\ell}} > x_N\}}$. We must verify the conditions of Lemma~\ref{L_iterated_integrals}. It is clear that for each $i,j$, $G^{ij}$ is adapted to $(\Fcal^{\{i,j\}}_t)_{t \ge 0}$ and, thanks to Assumption~\ref{ass2} and the uniform boundedness of $D_3B^i$, belongs to $\mathbb L$. Indeed, a brief calculation yields
\begin{equation} \label{eq_Gij_integrability}
\E\left[ ( G^{ij}_t )^{2p} \right] \le 2 C p! K(t)^p
\end{equation}
for any $p \in \N$ and $t \in \R_+$, where $C$ is a uniform bound on $B_3B^i$ and $K(t)$ comes from Assumption~\ref{ass2}. Moreover, using that the $(X^i,W^i)$ are mutually independent and $D_3B^i$ is adapted to $(\Fcal^{\{i\}}_t)_{t \ge 0}$, one verifies that \eqref{L_iterated_integrals_eq_0} holds whenever $V \subset [N]$ and $i \in V$, $ j \notin V$.

Lemma~\ref{L_iterated_integrals} now tells us that the expectation in \eqref{step3result} vanishes whenever the subset $K=\{i_{01},\ldots,i_{0\kappa}\}$ and tuples $\bm i_1,\ldots,\bm i_\alpha, \bm j_1, \ldots, \bm j_\alpha$ satisfy at least one of the conditions \ref{L_iterated_integrals_1}--\ref{L_iterated_integrals_2} of the lemma. Thanks to Lemma~\ref{count}, for each $\kappa \in [N]$ there are at most
\[
\binom{N}{\kappa} \kappa (\kappa+1) \cdots (\kappa+S-1) (\kappa+S) N^{S-1}
\]
terms for which this is not the case, where we write $S = k_1 + \ldots + k_{\alpha}$. We claim, and will prove below, that each of these nonzero terms admits the bound
\begin{equation} \label{eq_bound_on_nonzero_terms}
\left| \mathbb{E} \left[ \prod_{\ell=1}^{\kappa} \bm1_{\{X_{T}^{i_{0\ell}} > x_N\}} \prod_{\beta = 1}^{\alpha} I_{{\bm i_{\beta}},{\bm j_{\beta}}}^{N}(T_{\beta - 1}, T_\beta) \right] \right|
\le
C \lceil \log N \rceil^{S}\left(\frac{C}{N}\right)^{\kappa (1 - 1/\log N)}
\end{equation}
for a constant $C$ that does not depend on $N$. Combining these two facts we upper bound \eqref{step3result} by
\begin{align*}
&\frac{1}{N^{k_1 + \cdots + k_\alpha}} \sum_{\kappa=1}^N  \binom{N}{\kappa} \kappa (\kappa+1) \cdots (\kappa+S) N^{S-1} C \lceil \log N \rceil^{S}\left(\frac{C}{N}\right)^{\kappa (1 - 1/\log N)} \\
&\quad = \frac{C \lceil \log N \rceil^{S}}{N} \sum_{\kappa=1}^N  \binom{N}{\kappa} \kappa (\kappa+1) \cdots (\kappa+S) \left(\frac{C}{N}\right)^{\kappa (1 - 1/\log N)}.
\end{align*}
Thanks to Lemma~\ref{combinatoryinequality}, this is in turn bounded by
\[
\frac{C \lceil \log N \rceil^{S}}{N} (S+2)(S+1)^{2(S+1)}e^{Ce}(Ce)^{S+1},
\]
which tends to zero as $N \to \infty$. Tracing backwards, we deduce that \eqref{step3result} and hence \eqref{step3expansion} tends to zero as well. This is true for any choice of $(k_1,\ldots,k_\alpha)$, showing that \eqref{step2result} tends to zero. As a result, we see from \eqref{AnalyticStep2Claim} that $\limsup_{N\to\infty} |A^N_\alpha| \le \alpha \varepsilon$ and thus $A^N_\alpha \to 0$ since $\varepsilon > 0$ was arbitrary. We recall from Step~1 that it was enough to obtain this for any $\alpha \in [n]$ in order to prove the theorem.

It still remains to establish \eqref{eq_bound_on_nonzero_terms}. Applying H\"older's inequality with exponents $p_N=\log N$ and $q_N = (1-1/\log N)^{-1}$ gives
\begin{align*}
\left| \mathbb{E} \left[ \prod_{\ell=1}^{\kappa} \bm1_{\{X_{T}^{i_{0\ell}} > x_N\}} \prod_{\beta = 1}^{\alpha} I_{{\bm i_{\beta}},{\bm j_{\beta}}}^{N}(T_{\beta - 1}, T_\beta) \right] \right|
&\le
\mathbb{E} \left[ \prod_{\ell=1}^{\kappa} \bm1_{\{X_{T}^{i_{0\ell}} > x_N\}} \right]^{\frac{1}{q_N}} \prod_{\beta = 1}^{\alpha} \left\| I_{{\bm i_{\beta}},{\bm j_{\beta}}}^{N}(T_{\beta - 1}, T_\beta) \right\|_{\alpha p_N} \\
&= \P\left( X_T > x_N \right)^{\kappa/q_N} \prod_{\beta = 1}^{\alpha} \left\| I_{{\bm i_{\beta}},{\bm j_{\beta}}}^{N}(T_{\beta - 1}, T_\beta) \right\|_{\alpha p_N}.
\end{align*}
Using Lemma~\ref{lpestimationiterated} with $k = k_{\beta}$ and $M_{t}^{\ell} = \int_0^{t}G^{i_{\beta\ell},j_{\beta\ell}}_s dW_s^{i_{\beta\ell}}$ for $\ell \in [k_{\beta}]$ we bound
\[
\left\|  I_{{\bm i_{\beta}},{\bm j_{\beta}}}^{N}(T_{\beta - 1}, T_\beta)  \right\|_p\leq (4\sqrt{p})^{k_\beta} 2^{k_\beta(k_\beta-1)/4}\prod_{\ell=1}^{k_\beta}\|\langle M^{\ell} \rangle_{T_{\beta-1},T_\beta}^{1/2}\|_{2^{\ell}p}
\]
for any $p \in \N$. Moreover, from H\"older's inequality and \eqref{eq_Gij_integrability}, and recalling also that $T_\beta - T_{\beta-1} ~=~ T/n$ (see Step~1), we get
\begin{align*}
\|\langle M^{\ell} \rangle_{T_{\beta-1},T_\beta}^{1/2}\|_{2^{\ell}p}
&\le
\left( \frac{T}{n} \right)^{\frac{1}{2} - \frac{1}{2^\ell p}} \left( \int_{T_{\beta-1}}^{T_\beta} \E\left[ ( G^{ij}_s )^{2^\ell p} \right] ds \right)^{1/(2^\ell p)} \\
&\le \left( \frac{T}{n} \right)^{1/2} (2 C (2^{\ell-1}p)! )^{1/(2^\ell p)} \sup_{t \le T} K(t)^{1/2} \\
&\le C \sqrt{p},
\end{align*}
where in the last step we used Stirling's approximation and where $C$ (which as per our conventions may change from one occurrence to the next) does not depend on $p$ or $N$. We deduce that
\[
\left\|  I_{{\bm i_{\beta}},{\bm j_{\beta}}}^{N}(T_{\beta - 1}, T_\beta)  \right\|_p\leq C (4\sqrt{p})^{k_\beta} \prod_{\ell=1}^{k_\beta} \sqrt{p} = C p^{k_\beta}.
\]
Choosing $p = \alpha \lceil p_N \rceil$ we apply the above bounds to obtain
\begin{align*}
\left| \mathbb{E} \left[ \prod_{\ell=1}^{\kappa} \bm1_{\{X_{T}^{i_{0\ell}} > x_N\}} \prod_{\beta = 1}^{\alpha} I_{{\bm i_{\beta}},{\bm j_{\beta}}}^{N}(T_{\beta - 1}, T_\beta) \right] \right|
&\le \P\left( X_T > x_N \right)^{\kappa/q_N} \prod_{\beta = 1}^{\alpha} C (\alpha \lceil p_N \rceil)^{k_\beta} \\
&= C \lceil p_N \rceil^{S}\, \P\left( X_T > x_N \right)^{\kappa/q_N}.
\end{align*}
All that remains in order to establish \eqref{eq_bound_on_nonzero_terms} is to show that $\P( X_T > x_N) \le C/N$. But this follows from the fact that $(1 - \P(X_T > x_N))^N = \P(\max_{i \le N} X^i_T \le a^N_T x + b^N_T) \to \Gamma_T(x) > 0$ by assumption, so that $N \P(X_T > x_N) \le - N \log(1-\P(X_T > x_N)) \le C$ for some constant $C$ that does not depend on $N$. This completes the proof of \eqref{eq_bound_on_nonzero_terms}, and of the theorem.

\bibliography{references}

\begin{thebibliography}{10}

\bibitem{MR1080535}
G.~Ben~Arous and M.~Brunaud.
\newblock M\'{e}thode de {L}aplace: \'{e}tude variationnelle des fluctuations
  de diffusions de type ``champ moyen''.
\newblock {\em Stochastics Stochastics Rep.}, 31(1-4):79--144, 1990.

\bibitem{BOCA16}
L.~Bo and A.~Capponi.
\newblock Systemic risk in interbanking networks.
\newblock {\em SIAM J. Financ. Math.}, 6(1):386--424, 2015.

\bibitem{CAKR91}
E.~Carlen and P.~Kree.
\newblock Lp estimates on iterated stochastic integrals.
\newblock {\em Ann. Probab.}, 19(1):354--368, 1991.

\bibitem{CFS13}
R.~Carmona, J.~Fouque, and L.~Sun.
\newblock Mean field games and systemic risk.
\newblock {\em Communications in Mathematical Sciences}, 13(4):911--933, 2015.

\bibitem{COEL15}
S.~N. Cohen and R.~J. Elliott.
\newblock {\em Stochastic calculus and applications}, volume~5 of {\em
  Probability and its Applications}.
\newblock Springer, Cham, second edition, 2015.

\bibitem{MR885876}
D.~A. Dawson and J.~G\"{a}rtner.
\newblock Large deviations from the {M}c{K}ean-{V}lasov limit for weakly
  interacting diffusions.
\newblock {\em Stochastics}, 20(4):247--308, 1987.

\bibitem{HF10}
L.~de~Haan and A.~Ferreira.
\newblock {\em Extreme Value Theory: An Introduction}.
\newblock Springer Series in Operations Research and Financial Engineering.
  Springer New York, 2007.

\bibitem{MR1894767}
E.~R. Fernholz.
\newblock {\em Stochastic portfolio theory}, volume~48 of {\em Applications of
  Mathematics (New York)}.
\newblock Springer-Verlag, New York, 2002.
\newblock Stochastic Modelling and Applied Probability.

\bibitem{MR968996}
J.~G\"{a}rtner.
\newblock On the {M}c{K}ean-{V}lasov limit for interacting diffusions.
\newblock {\em Math. Nachr.}, 137:197--248, 1988.

\bibitem{MR173726}
J.~Ginibre.
\newblock Statistical ensembles of complex, quaternion, and real matrices.
\newblock {\em J. Mathematical Phys.}, 6:440--449, 1965.

\bibitem{JA19}
J.-F. Jabir.
\newblock Rate of propagation of chaos for diffusive stochastic particle
  systems via girsanov transformation.
\newblock {\em arXiv: Probability}, 2019.
\newblock \url{ https://arxiv.org/abs/1907.09096}.

\bibitem{JK21}
Y.~Jiao and N.~Kolliopoulos.
\newblock Well-posedness of a system of sdes driven by jump random measures.
\newblock {\em arXiv: Probability}, 2021.
\newblock \url{ https://arxiv.org/abs/2102.03918}.

\bibitem{MR3327514}
B.~Jourdain and J.~Reygner.
\newblock Propogation of chaos for rank-based interacting diffusions and long
  time behaviour of a scalar quasilinear parabolic equation.
\newblock {\em Stoch. Partial Differ. Equ. Anal. Comput.}, 1(3):455--506, 2013.

\bibitem{MR3340241}
B.~Jourdain and J.~Reygner.
\newblock Capital distribution and portfolio performance in the mean-field
  {A}tlas model.
\newblock {\em Ann. Finance}, 11(2):151--198, 2015.

\bibitem{MR3773380}
P.~Kolli and M.~Shkolnikov.
\newblock S{PDE} limit of the global fluctuations in rank-based models.
\newblock {\em Ann. Probab.}, 46(2):1042--1069, 2018.

\bibitem{MR3841406}
D.~Lacker.
\newblock On a strong form of propagation of chaos for {M}c{K}ean-{V}lasov
  equations.
\newblock {\em Electron. Commun. Probab.}, 23:Paper No. 45, 11, 2018.

\bibitem{LA22}
D.~Lacker.
\newblock Hierarchies, entropy, and quantitative propagation of chaos for mean
  field diffusions.
\newblock {\em arXiv: Probability}, 2021.
\newblock \url{ https://arxiv.org/abs/2105.02983}.

\bibitem{LEDThesis}
S.~Ledger.
\newblock {\em Particle systems and stochastic PDEs on the half-line}.
\newblock PhD thesis, University of Oxford, 2016.

\bibitem{MR221595}
H.~P. McKean, Jr.
\newblock A class of {M}arkov processes associated with nonlinear parabolic
  equations.
\newblock {\em Proc. Nat. Acad. Sci. U.S.A.}, 56:1907--1911, 1966.

\bibitem{MR2364939}
S.~I. Resnick.
\newblock {\em Extreme values, regular variation and point processes}.
\newblock Springer Series in Operations Research and Financial Engineering.
  Springer, New York, 2008.
\newblock Reprint of the 1987 original.

\bibitem{MR1986426}
B.~Rider.
\newblock A limit theorem at the edge of a non-{H}ermitian random matrix
  ensemble.
\newblock volume~36, pages 3401--3409. 2003.
\newblock Random matrix theory.

\bibitem{HDS15}
P.~Rigollet.
\newblock {\em Graduate lecture notes in High - Dimensional Statistics, Chapter
  1}.
\newblock MIT OpenCourseWare, Massachusetts Institute of Technology, \url{
  https://ocw.mit.edu/courses/18-s997-high-dimensional-statistics-spring-2015/resources/mit18\_s997s15\_chapter1/},
  2015.

\bibitem{MR2914770}
M.~Shkolnikov.
\newblock Large systems of diffusions interacting through their ranks.
\newblock {\em Stochastic Process. Appl.}, 122(4):1730--1747, 2012.

\bibitem{MR1108185}
A.-S. Sznitman.
\newblock Topics in propagation of chaos.
\newblock In {\em \'{E}cole d'\'{E}t\'{e} de {P}robabilit\'{e}s de
  {S}aint-{F}lour {XIX}---1989}, volume 1464 of {\em Lecture Notes in Math.},
  pages 165--251. Springer, Berlin, 1991.

\bibitem{MR1257246}
C.~A. Tracy and H.~Widom.
\newblock Level-spacing distributions and the {A}iry kernel.
\newblock {\em Comm. Math. Phys.}, 159(1):151--174, 1994.

\end{thebibliography}
\bibliographystyle{abbrv}

\end{document}